\documentclass[a4paper,10pt]{article}
\pdfoutput=1
\usepackage[english]{babel}
\usepackage{amssymb}
\usepackage{amsmath}
\usepackage{amsthm}
\usepackage{graphicx}
\usepackage{algorithm2e}
\usepackage{fancyhdr}
\newfont{\bssten}{cmssbx10}
\newfont{\bssnine}{cmssbx10 scaled 900}
\newfont{\bssdoz}{cmssbx10 scaled 1200}

\usepackage{latexsym,amssymb,amsthm,amsxtra}
\usepackage{amsmath}
\usepackage{stmaryrd}
\usepackage{mathrsfs}
\usepackage{tikz} 
\usepackage{pgf} 
{\bf}{\it}
\newtheorem{theorem}{Theorem}
\newtheorem{definition}{Definition}
\newtheorem{lemma}{Lemma}
\newtheorem{remark}{Remark}
\newtheorem{proposition}{Proposition}
\newtheorem{corollary}{Corollary}
\newtheorem{ex}{Example}

\makeatletter
\DeclareRobustCommand{\cev}[1]{%
  \mathpalette\do@cev{#1}%
}
\newcommand{\do@cev}[2]{%
  \fix@cev{#1}{+}%
  \reflectbox{$\m@th#1\vec{\reflectbox{$\fix@cev{#1}{-}\m@th#1#2\fix@cev{#1}{+}$}}$}%
  \fix@cev{#1}{-}%
}
\newcommand{\fix@cev}[2]{%
  \ifx#1\displaystyle
    \mkern#23mu
  \else
    \ifx#1\textstyle
      \mkern#23mu
    \else
      \ifx#1\scriptstyle
        \mkern#22mu
      \else
        \mkern#22mu
      \fi
    \fi
  \fi
}

\makeatother

\def\epsilon{\varepsilon}

\usepackage{anysize}
\marginsize{3cm}{2.5cm}{2.5cm}{2.5cm} 
\newcommand{\be}{\begin{equation}}
\newcommand{\ee}{\end{equation}}

\newcommand\suitekk[1]{\left(#1\right)_{n\ge -k}}
\def\W{\mathbb W}

\newcommand{\norm}[1]{\left|#1\right|}
\newcommand\td[1]{\overline{#1}}

\def\R{{\mathbb R}}
\def\Z{{\mathbb Z}}

\def\N{{\mathbb N}}

\def\maC{{\mathscr{C}}}

\def\T{{\mathcal{T}}}

\def\cv{\cev}

\def\maA{{\mathscr{A}}}
\def\maV{{\mathcal{V}}}
\def\maB{{\mathscr{B}}}
\def\maD{{\mathcal{D}}}
\def\maH{{\mathcal{H}}}

\def\maM{{\mathscr{M}}}
\def\maP{{\mathcal{P}}}
\def\maE{{\mathcal{E}}}
\def\maI{{\mathcal{I}}}

\def\I{{\mathbb{I}}}
\def\v{{\--}}
\def\pv{{\not\!\!\--}}

\def\T2a{{\tau_{2\vec \alpha^{+}}}}
\def\t2a{{t_{2\vec \alpha^{+}}}}
\newcommand\suite[1]{\left(#1\right)_{n\in\N}}
\newcommand\suitez[1]{\left(#1\right)_{n\in\Z}}
\newcommand\suitei[1]{\left(#1\right)_{i\in\N_+}}

\def\\xi{{\cal{\xi}}}

\def\({{\Bigl(}}
\def\){{\Bigr)}}
\newcommand\bp{{\mathbb P^0}}
\newcommand\bpr[1]{{\mathbb P^0}\left[#1\right]}
\newcommand\pr[1]{{\mathbb P}\left[#1\right]}

\def\square{\ifmmode\sqr\else{$\sqr$}\fi}
\def\sqr{\vcenter{
         \hrule height.1mm
         \hbox{\vrule width.1mm height2.2mm\kern2.18mm\vrule width.1mm}
         \hrule height.1mm}}                  

\newcommand\ccc{\circledcirc}

\def\Mfcfs{M_{\textsc{fcfm}}}
\def\Mlcfs{M_{\textsc{lcfm}}}
\def\Wfcfs{W_{\textsc{fcfm}}}

{\end{list}}

\def\emptystate{\mathbf{\emptyset}}

\newcommand\gre{\mathbf{e}}
\def\maL{\mathcal L}
\def\maM{\mathcal M}

\begin{document}
\title{On the sub-additivity of stochastic matching}
\author{P. Moyal, A. Bu$\check{\mbox{s}}$i\'c and J. Mairesse\\
\small{Université de Lorraine, INRIA and CNRS/Universit\'e Pierre et Marie Curie}}
\maketitle

\begin{abstract}
\quad We consider a stochastic matching model with a general compatibility graph, as introduced in \cite{MaiMoy16}. 
We prove that most common matching policies (including FCFM, priorities and random) satisfy a particular sub-additive property, which we exploit to show in many cases, 
the coupling-from-the-past to the steady state, using a backwards scheme {\em \`a la} Loynes. 
We then use these results to explicitly construct perfect bi-infinite matchings, and to build a perfect simulation algorithm in the case where the buffer of the system is finite. 
\end{abstract}
%
\section{Introduction}
\label{sec:intro}
We consider a general stochastic matching model (GM), as introduced in \cite{MaiMoy16}: 
items of various classes enter a system one by one, to be matched by couples. Two items are compatible if and only if their classes are adjacent in a compatibility graph $G=(\maV,\maE)$ that is fixed beforehand. The classes of the entering items are drawn following a prescribed probability measure on $\maV$. 
This model is a variant of the Bipartite Matching model (BM) introduced in \cite{CKW09}, see also \cite{AW12}, in which case the compatibility graph is bipartite of bipartition 
$\maV=\maV_1 \cup \maV_2$. Along 
the various natural applications of this model, the nodes of $\maV_1$ and $\maV_2$ represent respectively classes of {customers} and {servers}, 
{kidneys} and {patients}, blood {givers} and blood {receivers}, {houses} and {applicants}, and so on. 
The items are matched by couples of $\maV_1\times\maV_2$, and also {arrive} 
by couples of $\maV_1\times\maV_2$. See \cite{ABMW17}, and reference therein. 
The extension of the BM to the context of general (instead of bipartite) compatibility graphs, leading to the GM model, 
allows naturally to take into account applications for which there is no bipartition of the classes of items, 
such as assemble-to-order systems, dating websites, car-sharing and cross-kidney transplants. 

An important generalization of the BM is the so-called {Extended Bipartite Matching} model (EBM) introduced in \cite{BGMa12}, where this independent assumption is relaxed. 
Possible entering couples are element of a bipartite {arrival graph} on the bipartition $\maV_1\cup\maV_2$. 
Importantly, notice that the GM can in fact be seen as a particular case of EBM, taking the bipartite double cover of $G$ as compatibility graph, and duplicating arrivals with a copy of an item of the same class.

Coming back to GM models, \cite{MaiMoy16} investigated the form of the {stability region} of the model, namely the set of probability measures on $\maV$ rendering the corresponding system stable. Partly relying on the aforementioned connection between GM and EBM, and the results of \cite{BGMa12}, \cite{MaiMoy16} shows that the stability region is always included in the set of measures satisfying the natural condition (\ref{eq:Ncond}) below. The form of the stability region is then heavily dependent on the 
geometry of the compatibility graph, and on the {matching policy}, i.e. the rule of choice of a match for an entering item whenever several possible matches are possible. A matching policy is said to have a {maximal} stability region for $G$ if the system is stable for any measure satisfying (\ref{eq:Ncond}). 
It is shown in \cite{MaiMoy16} that a GM on a bipartite $G$ is never stable, that a designated class of graphs 
(the complete $k$-partite graphs for $k\ge 3$, see below) make the stability region maximal for all matching policies, and that the policy `Match the Longest' always has a maximal stability region for a non-bipartite $G$. 
Applying fluid instability arguments to a continuous-time version of the GM, \cite{MoyPer17} show that, aside for a very particular class of graphs, whenever $G$ is not complete $k$-partite there {always} exists a policy of the strict priority type that does not have a maximal stability region, and that the `Uniform' random policy (natural in the case where no information is available to the entering items on the state of the system) never has a maximal stability region, 
thereby providing a partial converse of the result in \cite{MaiMoy16}. Following the approach of \cite{ABMW17} (see also \cite{GW20}), \cite{MBM21} shows that the GM model also enjoys a product form in steady state under the FCFM (`First Come, First Matched') policy. In recent years, the GM model was studied along various other angles, among which: Optimization \cite{NS17}, optimal control \cite{GurWa}, stability of matching models on hypergraphs \cite{RM21}, of graphs with self-loops \cite{BMMR21}, or models with reneging, see \cite{BDPS11,JMRS23}. Recently, GM models have been shown to share remarkable similarities with order-independent loss queues, see \cite{Com21}, and to exhibit performance paradoxes (non-monotonicity of the performance of the system with respect to the number of edges), in \cite{CDFB22}. See also \cite{CMB21,BMM23} for a in-depth study of the form of the set defined by \eqref{eq:Ncond} and thereby, for an explicit construction of stabilizing matching rates for FCFM and ML, in function of the general graph geometry.

In this work, we use coupling-from-the-past techniques to construct the stationary state of stable BM models, 
for a set of matching policies that includes FCFM and ML (in view of the maximality result announced above). 
It is well known since the pioneering works of Loynes \cite{Loynes62} and then Borovkov \cite{Bor84}, that backwards schemes and specifically strong backwards coupling convergence, can lead to an explicit construction  
of the stationary state of the system under consideration within its stability region. 
One can then use pathwise representations to compare systems in steady state, via the stochastic ordering of a given performance metric (see Chapter 4 of \cite{BacBre02} on such comparison results for queues). 
Moreover, we know since the seminal work of Propp and Wilson \cite{PW96} that coupling-from-the-past algorithms (which essentially use backwards coupling convergence) provide a powerful tool for simulating the steady state of the system, even whenever the latter distribution is not know in closed form. We aim at achieving such constructive results for the general matching model: under various conditions, we derive a stationary version of the system under general stationary ergodic assumptions, 
via a stochastic recursive representation of the system on the canonical space of its bi-infinite input. 
For this, we first observe that most usual matching policies (including FCFM, the optimal 'Match the Longest' policy, and - possibly randomized - priorities) satisfy a remarkable sub-additive property, 
which allows to build the appropriate backwards scheme to achieve this explicit construction. 
These results lead to a result of backwards coupling convergence, to a unique stationary state. 
Then, we apply this coupling result in two directions: First, we deduce the construction of a unique (up to the natural parity of the model, in a sense that will be specified below) stationary bi-infinite perfect matching. This result extends the results of \cite{AW12,ABMW17} to general graphs, and to a wide class of matching policies. Second, we use this backwards coupling result to construct a perfect simulation algorithm, in the case where the system capacity if finite. 

The paper is organized as follows. In Section \ref{sec:model} we introduce and formalize our model. 
The sub-additivity property of a wide class of matching policies is shown in Section \ref{sec:subadd}. Our coupling result is then presented and proven in Section \ref{sec:coupling}. In Section \ref{sec:matching} we show how these results 
can be used to construct (unique) perfect bi-infinite matchings of the incoming items. Our perfect simulation algorithm for finite-capacity systems is developed in Section \ref{sec:PW}.

\section{The model}
\label{sec:model}

\subsection{General notation}
\label{subsec:notation} 
Denote by $\R$ the real line, by $\N$ the set of non-negative integers and by $\N_+$, the subset of positive integers. For any two integers $m$ and $n$, denote by $\llbracket m,n \rrbracket=[m,n] \cap \N$. 
For any finite set $A$, let $S_A$ be the group of permutations of $A$, and for all permutation $s \in S_A$ and any $a\in A$, let $s[a]$ be the image of $a$ by $s$. 
Let $A^*$ (respectively, $A^{**}$) be the set of finite (resp., infinite) words over the alphabet $A$. 
Denote by $\emptyset$, the empty word of $A^*$. 
For any word $w \in A^*$ and any subset $B$ of $A$, we 
let $\norm{w}_B$ be the number of occurrences of elements of $B$ in $w$. 
For any  letter $a\in A$, we denote 
$\norm{w}_a:=\norm{w}_{\{a\}}$, and for any finite word $w$ we let $\norm{w}=\sum_{a\in A} \norm{w}_a$ be the
{\em length} of $w$. For a word $w \in A^*$ of length $\norm{w}=q$, we write $w=w_1w_2...w_q$, i.e. $w_{i}$ is the $i$-th letter 
of the word $w$. In the proofs below, we understand the word $w_1...w_k$ as $\emptyset$ whenever $k=0$. 
Also, for any $w \in A^*$ and any $i\in \llbracket 1,\norm{w} \rrbracket$, we denote by $w_{[i]}$, the word of length $\norm{w}-1$ obtained from $w$ by deleting 
its $i$-th letter. We let $[w]:=(\norm{w}_a)_{a\in A}\in \N^A$ be the {\em commutative
image} of $w$. Finally, a {\em suffix} of the word $w=w_1...w_k$ is a word $w_j...w_k$ obtained by deleting the first $j-1$ letters of $w$, for some $j \in \llbracket 1,k \rrbracket$.  
For any $p \in \N_+$, a vector $x$ in the set $A^p$ is denoted $x=(x(1),...,x(p))$. For any $i\in\llbracket 1,p \rrbracket$, we denote by $\gre_i$ the $i$-th vector of the canonical basis of $\R^p$, 
i.e. $\gre_i(j)=\delta_{ij}$ for any $j\in \llbracket 1,p \rrbracket$. The $\ell_1$ norm of $\R^p$ is denoted $\parallel . \parallel$.   

Consider a simple graph $G=(\maV,\maE)$, where $\maV$ denotes the set of nodes, and $\maE \subset \maV\times \maV$ is the set of edges. 
We use the
notation $u \v v$ for $(u,v) \in \maE$
and $u \pv v$ for $(u,v) \not\in \maE$. 
For $U \subset \maV$, we define $U^c = \maV \setminus U$ and
\[
\maE(U) = \{v \in \maV\,:\, \exists u \in U, \ u
\-- v\}\:.
\]
An {\em independent set} of $G$ is a non-empty subset $\maI \subset \maV$ 
which does not include any pair of neighbors, {\em i.e.} $\bigl(\forall i\neq j \in \maI, \ i \pv j\bigr)$. 
Let $\I(G)$ be the set of independent sets of $G$. An independent set $\maI$ is said to be {\em maximal} if $\maI \cup \{j\} \not\in \I(G)$ for any $j \not\in \maI$.

\subsection{Formal definition of the model}
\label{subsec:model}
We consider a {\em general stochastic matching model}, as was defined in \cite{MaiMoy16}: items enter one by one a system, and each of them belongs to a 
determinate class. The set of classes is denoted by $\maV$, and identified with $\llbracket 1,|\maV|\rrbracket$. We fix a connected simple graph $G=(\maV,\maE)$ having set of nodes $\maV$, termed {\em compatibility graph}. 
Upon arrival, any incoming item of class, say, $i \in \maV$ is either matched with an item present in the buffer, of a class $j$ such that 
$i \v j$, if any, or if no such item is available, it is stored in the buffer to wait for its match. 
Whenever several possible matches are possible for an incoming item $i$, a {\em matching policy} $\phi$ decides what is the match of $i$ without ambiguity. 
Each matched pair departs the system right away. 

We assume that the successive classes of entering items, and possibly their choices of match, are random.  
We fix a probability space $(\Omega,\mathcal F,\mathbb P)$ on which all random variables (r.v.'s, for short) are defined, and view, throughout, 
the input as a bi-infinite sequence $\left(V_n,\Sigma_n\right)_{n\in\Z}$ that is defined as follows:  
first, for any $n \in \Z$ we let $V_n \in \maV$ denote the class of the $n$-th incoming item. 
Second, we introduce the set   
\begin{equation*}
\mathcal S = S_{\maE(1)} \times ... \times S_{\maE(|\maV|)}, 
\end{equation*}
in other words for any $\sigma = \left(\sigma(1),...,\sigma(|\maV|)\right) \in \mathcal S$ and $i \in \maV$, 
$\sigma(i)$ is a permutation of the classes of items that are compatible with $i$ (which are identified with their indexes in $\llbracket 1,|\maV| \rrbracket$). 
Any array of permutations $\sigma \in \mathcal S$ is called {\em list of preferences}. 
For any $n \in \Z$, we let $\Sigma_n$ denote the list of preferences at time $n$, i.e. if $\Sigma_n=\sigma$ and $V_n=v$, then the permutation $\sigma(v)$ represents the order of preference of the entering $v$-item 
at $n$, among the classes of its possible matches. 
Throughout this work, we suppose that the sequence $\suitez{\left(V_{n},\Sigma_{n}\right)}$ is iid from the distribution $\mu \otimes \nu_\phi$ on $\mathcal V\times\mathcal S$. 

We also assume that $\mu$ has full support $\maV$ (we write $\mu \in \maM(\maV)$). 
Then, the matching policy $\phi$ will be formalized by an operator mapping the system state onto the next one, given the class of the entering 
item and the list of preferences at this time. The matching policies we consider are presented in detail in Section \ref{subsec:pol}. 

Altogether, the compatibility graph $G$, the matching policy $\phi$ and the measure $\mu$ fully specify the model, which we denote for short general matching (GM) model associated with $(G,\mu,\phi)$. 

\subsection{State spaces}
\label{subsec:state}
Fix the compatibility graph $G=(\maV,\maE)$ until the end of this section. 
Fix an integer $n_0 \ge 1$, and two realizations $v_1,...v_{n_0}$ of $V_1,...,V_{n_0}$ and $\sigma_1,...,\sigma_{n_0}$ of $\Sigma_1,...,\Sigma_{n_0}$. 
Define the two words $z\in \maV^*$ and $\varsigma\in\mathcal S^*$ by $z:=v_1...v_{n_0}$ and $\varsigma:=\sigma_1...\sigma_{n_0}$. 
Then, for any matching policy $\phi$ there exists a unique {\em matching} of the word $z$ associated to $\varsigma$, that is, a graph having set of nodes 
$\left\{v_1,...,v_{n_0}\right\}$ and whose edges represent the matches performed in the system until time $n_0$, if the successive arrivals are given by $z$ and the lists of preferences by $\varsigma$.  
This matching is denoted by $M_\phi(z,\varsigma)$. 
The state of the system is then defined as the word $W_\phi(z,\varsigma)\in \maV^*$, whose letters are the classes of the unmatched items at $n_0$, 
i.e. the isolated vertices in the matching $M^{\phi}(z,\varsigma)$, in their order of arrivals. The word $W_\phi(z,\varsigma)$ is called {\em queue detail} at time $n_0$. 
Then any admissible queue detail belongs to the set 
\begin{equation}
\mathbb W = \Bigl\{ w\in \maV^*\; : \; \forall  (i,j) \in \maE, \; |w|_i|w|_j=0  \Bigr\}.\label{eq-ncss}
     \end{equation} 
As will be seen below, depending on the service discipline $\phi$ we can also restrict the available information on the state of the system at time $n_0$, to a vector only keeping track of 
the number of items of the various classes remaining unmatched at $n_0$, that is, of the number of occurrences of the various letters of the alphabet $\maV$ in the word $W_\phi(z,\varsigma)$.    
This restricted state thus equals the commutative image of $W^{\phi}(z,\varsigma)$, and is called {\em class detail} of the system. It takes values in the set  
\begin{equation}
\mathbb X = \Bigl\{x \in \N^{|\maV|}\,:\,x(i)y(j)=0\mbox{ for any }(i,j)\in \maE\Bigl\}=\Bigl\{\left[w\right];\,w \in \mathbb W\Bigl\}.\label{eq-css}
\end{equation} 

\subsection{Matching policies}
\label{subsec:pol}
We now present and define formally, the set of matching policies which we consider. Notice that, contrary to various policies addressed in 
\cite{GurWa,NS17}, or in \cite{CBD19} regarding bipartite matching models, we address here only {\em greedy} policies; namely, it is never the case that two compatible items 
are stored together in the system. To the contrary, the (possibly non-greedy) matching policies addressed in the three aforementioned references allow the possibility of not executing 
a possible match, to wait for a more profitable future match. 

\begin{definition}
A matching policy $\phi$ is said {\em admissible} if the choice of match of an incoming item depends 
{solely} on the queue detail and the list of preferences drawn upon arrival. 
\end{definition} 
An admissible matching policy can be formally characterized by an action $\odot_{\phi}$ of $\maV\times \mathcal S$ on $\mathbb W$, defined as follows: 
if $w$ is the queue detail at a given time and the input is augmented by the arrival of a couple $(v,\sigma) \in \maV\times \mathcal S$ at that time, then the new queue detail $w'$ and $w$ satisfy the relation 
\begin{equation}
\label{eq:defodot}
w'= w\odot_{\phi} (v,\sigma).
\end{equation}

\subsubsection{Matching policies that depend on the arrival times}  
We first introduce two matching policies that depend on the arrival dates of the items. 
In 'First Come, First Matched' ({\sc fcfm}) the map  $\odot_{\textsc{fcfm}}$ is clearly independent of the list of preferences $\sigma$. 
It is given for all $w \in \mathbb W$ and all couples $(v,\sigma)$, by 
$$
w \odot_{\textsc{fcfm}} (v,\sigma) =
\left \{
\begin{array}{ll}
wv & \textrm{if } \; |w|_{\maE(v)} = 0;\\
w_{\left [\Phi(w,v)\right]} & \textrm{else, where }\Phi(w,v) = \arg\min \{|w_k|:\,k\in\maE(v)\},
\end{array}
\right .
$$
In 'Last Come, First Matched' ({\sc lcfm}) the updating map $\odot_{\textsc{lcfm}}$ is analog to $\odot_{\textsc{fcfm}}$, for 
$\Phi(w,v) = \arg\max\{|w_k|:\,k\in\maE(v)\}.$ 

\subsubsection{Class-admissible matching policies}
A matching policy $\phi$ is said to be {\em class-admissible} if it can be implemented 
by knowing only the class detail of the system. Let us define for any $v \in \maV$ and $x\in \mathbb X$,
\begin{equation*}
\maP(x,v) =\Bigl\{j\in \maE(v)\,:\,x\left(j\right) > 0\Bigl\},\label{eq:setP2}
\end{equation*}
the set of classes of available compatible items with the entering class $v$-item, if the class detail of the system is given by $x$. 
Then, a class-admissible policy $\phi$ is fully characterized by the probability distribution $\nu_\phi$ on $\mathcal S$, 
together with a mapping $p_\phi$ such that $p_\phi(x,v,\sigma)$ denotes the class of the match chosen 
by the entering $v$-item under $\phi$ for a list of preferences $\sigma$, in a system of class detail $x$ such that $\mathcal P(x,v)$ is non-empty. 
Then the arrival of $v$ and the draw of $\sigma$ from $\nu_\phi$ corresponds to the following action on
the class detail,
\begin{equation}
\label{eq:defccc}
x \ccc_{\phi} (v,\sigma) = \left \{
\begin{array}{ll}
x+\gre_v &\mbox{ if }\mathcal P(x,v)=\emptyset,\\
x-\gre_{p_\phi(x,v,\sigma)}&\mbox{ else}. 
\end{array}
\right .
\end{equation}

\begin{remark}
\label{rem:equiv}
\rm
As is easily seen, to any class-admissible policy $\phi$ corresponds an admissible policy, if one makes precise the rule of choice of match for the 
incoming items {\em within} the class that is chosen by $\phi$, in the case where more than one item of that class is present in the system. 
In this paper, we always make the assumption that within classes, the item chosen is always the {\em oldest} in line, i.e. we always apply a FCFM policy {\em within classes}.  
Under this convention, any class-admissible policy $\phi$ is admissible, that is, the mapping $\ccc_\phi$ from 
$\mathbb X\times (\maV \times \mathcal S)$ to $\mathbb X$ can be detailed into a map $\odot_{\phi}$ from 
$\mathbb W \times (\maV \times \mathcal S)$ to $\mathbb W$, as in (\ref{eq:defodot}), that is such that for any queue detail $w$ and any $(v,\sigma)$,
\[\left[w\odot_\phi (v,\sigma)\right] = \left[w\right]\ccc_\phi (v,\sigma).\]    
\end{remark}

\paragraph{Random policies.} 
In a random policy, the only information that is needed to determine the choice of match of an incoming item, is whether its various compatible classes have an empty queue or not. 
Specifically, the order of preference of each incoming item is drawn upon arrival following the prescribed probability distribution; then the considered item investigates its compatible classes in that order, 
until it finds one having a non-empty buffer, if any. The incoming item is then matched with an item of the latter class. In other words, a list of preferences $\sigma=\left(\sigma(1),...,\sigma\left(|\maV|\right)\right)$ is drawn from $\nu_\phi$ on $\mathcal S$, and we set   
\begin{equation}
p_{\phi}(x,v,\sigma)=\sigma(v)[k],\mbox{ where }k=\min \Bigl\{i \in
\maE(v)\,:\,\sigma(v)[i]\in \maP(x,v)\Bigl\}\label{eq:pphirandom}.
\end{equation}
In particular, the `Class-uniform' policy {\sc u} is such that $\nu_\phi$ is the uniform distribution on $\mathcal S$. 
In other words, for any $i \in \maV$ and any $j$ such that $j \v i$, $\sigma(i)[j]$ is drawn uniformly in $\maE(i)$, that is, 
the class of the match of the incoming $i$-item is chosen uniformly among all compatible classes having a non-empty buffer.

\paragraph{Priority policies.} 
In a priority policy, for any $v\in \maV$ the order of preference of $v$ in $\maE(v)$ is deterministic. This is thus another particular case of random policy in which a list of preference $\sigma^0 \in \Sigma$ is 
fixed beforehand, in other words $\nu_\phi=\delta_{\sigma^0}$ and (\ref{eq:pphirandom}) holds for $\sigma:=\sigma^0$. 

\paragraph{`Match the Longest' and `Match the Shortest'}
In `Match the Longest' ({\sc ml}),  
the newly arrived item chooses an item of the compatible class that has the longest line. Ties are broken by a
uniform draw between classes having queues of the same maximal length. 
Formally, set for all $x$ and $v$ such that $\mathcal P(x,v) \ne \emptyset$, 
\begin{equation*}
L(x,v) =\max\left\{x(j)\,:\,j \in \maE(v)\right\}\,\quad\mbox{ and }\quad\,
\maL(x,v) =\left\{i\in \maE(v)\,:\,x\left(i\right)=L(x,v)\right\}\subset \maP(x,v).
\end{equation*}
Then, set $\nu_\phi$ as the uniform distribution on $\mathcal S$. 
If the resulting sample is $\sigma$, we have
\begin{equation*}
p_{\textsc{ml}}(x,v,\sigma) =\sigma(v)[k],\mbox{ where }k=\min \Bigl\{i \in \maE(v)\,:\,\sigma(v)[i]\in \maL(y,c)\Bigl\}.
\end{equation*}
Likewise, the `Match the Shortest' ({\sc ms}) policy is defined similarly to {\sc ml}, except that the shortest non empty queue is chosen instead of
the longest.

\subsection{Markov representation}
\label{subsec:Markov}

Fix a (possibly random) word 
$w \in \mathbb W$ and a word $\varsigma \in \mathcal S^*$ having the same length as $w$. Denote for all $n\ge 0$ by $W^{\{w\}}_n$ the buffer content at time $n$ 
(i.e. just before the arrival of item $n$) if the buffer content at time 0 was set to $w$, in other words 
\[W^{\{w\}}_n= W_\phi\left(wV_0...V_n\,,\,\varsigma \Sigma_0...\Sigma_n\right).\]
It follows from (\ref{eq:defodot}) that the buffer-content sequence is a Markov chain, since we have that  
\[\left\{\begin{array}{ll}
W^{\{w\}}_0 &= w;\\
W^{\{w\}}_{n+1} &=W^{\{w\}}_n \odot_\phi (V_n,\Sigma_n),\,n\in\N. 
\end{array}\right.\]

Second, we deduce from (\ref{eq:defccc}) that for any class-admissible matching policy $\phi$ (e.g. $\phi=\textsc{random}, \textsc{ml}$ or $\textsc{ms}$), 
for any initial conditions as above, the $\mathbb X$-valued sequence $\suite{X_n}$ of class-details is also Markov: for any initial condition $x \in \mathbb X$, 
\begin{equation}
\label{eq:recurW}
\left\{\begin{array}{ll}
X^{\{x\}}_0 &= x;\\
X^{\{x\}}_{n+1} &=X^{\{x\}}_n \ccc_\phi (V_n,\Sigma_n),\,n\in\N.
\end{array}\right.
\end{equation}

\section{Sub-additivity}
\label{sec:subadd} 
We show hereafter that most of the
models we have introduced above satisfy a sub-additivity property
that will prove crucial in our main result. 
\begin{definition}[Sub-additivity]
\label{def:subadd}
An admissible matching policy $\phi$ is said to be {\em sub-additive} if, 
for all $z',z''\in \maV^*$, for all $\varsigma',\varsigma''\in\mathcal S^*$ whose letters are drawn by $\nu_\phi$ and such that $|\varsigma'|=|z'|$ and $|\varsigma''|=|z''|$, we have that 
$
\left|W_\phi(z'z'',\varsigma'\varsigma'')\right| \leq \left|W_\phi(z',\varsigma')\right| + \left|W_\phi(z'',\varsigma'')\right|. 
$
\end{definition}

\subsection{Non-expansiveness}
\label{subsec:nonexp}  
In the framework of stochastic recursions, the {\em non-expansiveness} property with respect to the $\ell_1$-norm, as introduced by Crandall and Tartar \cite{CT80}, amounts 
to the 1-Lipschitz property of the driving map of the recursion. Similarly, 
\begin{definition}[Non-expansiveness] 
A class-admissible policy $\phi$ is said {\em non-expansive} if 
for any $x$ and $x'$ in $\mathbb X$, any $v\in \maV$ and any $\sigma \in \mathcal S$ 
that can be drawn by $\nu_\phi$, 
\begin{equation}
\label{eq:defnonexp1}
\|x'\ccc_{\phi}(v,\sigma) - x\ccc_{\phi}(v,\sigma)\| \le \|x'-x\|.
\end{equation}
\end{definition} 


\begin{proposition}
\label{prop:nonexp1} 
Any random matching policy (in particular, priority and {\sc u}) is non-expansive. 
\end{proposition}

\begin{proof}
The result has been proven for priority and {\sc u} in \cite{MoyPer17}: this is precisely the inductive argument, respectively in the proofs of Lemma 4 and Lemma 7 therein. 
As is easily seen, the same argument can be generalized to any random policy $\phi$, once the list of preference that is drawn from $\nu_\phi$ is common to both systems.  
Indeed, the following consistency property holds: for any states $x$ and $x'$, any incoming item $v$ and any list of preferences $\sigma$ drawn from $\nu_\phi$, 
\begin{equation}
\label{eq:consist}
\biggl[\Bigl\{p_{\phi}(x,v,\sigma),p_{\phi}(x',v,\sigma)\Bigl\}\, \subset\, \maP(x,v) \cap \maP(x',v)\biggl]\quad \Longrightarrow \quad \biggl[p_{\phi}(x,v,\sigma) = p_{\phi}(x',v,\sigma)\biggl],     
\end{equation}
in other words, the choice of match of $v$ cannot be different in the two systems, if both options were available in both systems. 
The result follows for any random policy.  
\end{proof}

\begin{proposition}
\label{prop:nonexp2} 
{\sc ml} is non-expansive. 
\end{proposition}

\begin{proof}
The proof is similar to that for random policies, 
except for the consistency property (\ref{eq:consist}), which does not hold in this case. 
Specifically, an entering item can be matched
with items of two different classes in the two systems, whereas the
queues of these two classes are non-empty in both systems. 
Let us consider that case: specifically, a $v$-item enters the system, and for a common draw $\sigma$ according to the (uniform) distribution 
$\nu_{\textsc{ml}}$, we obtain $p_{\textsc{ml}}(x,v,\sigma)=k$ and $p_{\textsc{ml}}(x',v,\sigma)=k'$ for 
$\{k,k'\} \subset \maP(x,v) \cap \maP(x',v)$ and $k\ne k'$. 
Thus 
\begin{equation}
\label{eq:losers1} 
\|x'\ccc_{\textsc{ml}}(v,\sigma) -
x\ccc_{\textsc{ml}}(v,\sigma)\| =\sum_{i \ne k,k'}
|x(i)-x'(i)|+R, \end{equation} 
where $R=\left|(x(k)-1)-x'(k)\right|+\left|x(k')-\left(x'(k')-1\right)\right|.$ 
Then we have 
\[R=\left\{\begin{array}{ll}
           \left|x(k)-x'(k)\right|+\left|x(k')-x'(k')\right|-2 &\mbox{ if }x(k) > x'(k)\mbox{ and }x'(k') > x(k');\\
           \left|x(k)-x'(k)\right|+\left|x(k')-x'(k')\right| &\mbox{ if }x(k) \le x'(k)\mbox{ and }x'(k') > x(k');\\
           \left|x(k)-x'(k)\right|+\left|x(k')-x'(k')\right| &\mbox{ if }x(k) > x'(k)\mbox{ and }x'(k') \le x(k').
           \end{array}\right.\]
Observe that the case $x(k) \le x'(k)$ and $x'(k') \le x(k')$ cannot
 occur. Indeed, by the definition of {\sc ml} we have that $x(k') \le x(k)\,\mbox{ and }x'(k) \le x'(k')$, 
which would imply in turn that 
$x(k)=x(k')= x'(k) = x'(k').$ 
This is impossible since, in that case, under the common list of preferences $\sigma$ both systems would have
chosen the same match for the new $v$-item. 
As a conclusion, in view of (\ref{eq:losers1}), in all possible cases
we obtain that 
\begin{equation*}
\|x'\ccc_{\textsc{ml}}(v,\sigma) -
x\ccc_{\textsc{ml}}(v,\sigma)\|
\le \sum_{i \ne k,k'} |x(i)-x'(i)|+\left|x(k)-x'(k)\right|+\left|x(k')-x'(k')\right|
=\|x'-x\|,
\end{equation*} 
which concludes the proof.
\end{proof}

As the following counter-example demonstrates, the policy `Match the Shortest' is not non-expansive:

\begin{ex}[{\sc ms} is not non-expansive]
\label{ex:MS}
Take the graph of Figure \ref{Fig:paw} as a compatibility graph.
Set $x=(2,0,1,0)$, $x'=(1,0,2,0)$ and $v = 2$. Then we obtain that for all $\sigma$, 
$x\ccc_{\textsc{ms}}(v,\sigma) = (2,0,0,0)$ and $x'\ccc_{\textsc{ms}}(v,\sigma) = (0,0,2,0)$, 
and thus  $$\|x'\ccc_{\textsc{ms}}(v,\sigma) - x\ccc_{\textsc{ms}}(v,\sigma)\| = 4 > 2 = \|x'-x\|.$$
\end{ex}

\begin{figure}[h!]
\begin{center}
\begin{tikzpicture}[scale=0.8]
\draw[-] (2,3) -- (2,2);
\draw[-] (2,2) -- (1,1);
\draw[-] (2,2) -- (3,1);
\draw[-] (1,1) -- (3,1);
\fill (2,3) circle (2pt) node[right] {\small{1}} ;
\fill (2,2) circle (2pt) node[right] {\small{2}} ;
\fill (1,1) circle (2pt) node[below] {\small{3}} ;
\fill (3,1) circle (2pt) node[below] {\small{4}} ;
\end{tikzpicture}
\vspace*{-0.3cm}
\caption[smallcaption]{The `paw' graph.}
\label{Fig:paw}
\end{center}
\end{figure}
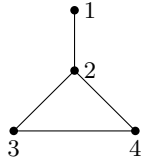

We have the following result,

\begin{proposition}
\label{prop:nonexpo}
Any non-expansive policy is sub-additive.
\end{proposition}

\begin{proof}
Fix a non-expansive matching policy $\phi$. Keeping the 
notations of Definition \ref{def:subadd}, let us define the two
arrays $(x_i)_{i=1,...,|v''|}$ and $\left(x'_i\right)_{i=1,...,|v''|}$ to be the class
details of the system at arrival times, starting respectively from
an empty system and from a system of buffer content $w'$, and
having a common input $\left(v''_i,\sigma''_i\right)_{i=1,...,|v''|}$, 
where $(\sigma''_i)_{i=1,...,|u''|}$ are drawn from $\nu_\phi$ on $\mathcal S$.   
In other words, we set 
\[\left\{\begin{array}{ll}
x_0 &= \mathbf 0;\\
x'_0 &= \left[w'\right]\\
\end{array}\right.
\quad \mbox{and} \quad 
\left\{\begin{array}{ll}
x_{n+1} &= x_{n} \ccc_{\phi} \left(v''_{n+1},\sigma''_{n+1}\right),\,n\in\left\{0,\dots,|v''|-1\right\};\\
x'_{n+1}&=x'_{n} \ccc_{\phi}\left(v''_{n+1},\sigma''_{n+1}\right),\,n\in\left\{0,\dots,|v''|-1\right\}.
\end{array}\right.\]
Applying (\ref{eq:defnonexp1}) at all $n$, we obtain by induction that for all $n
\in \left\{0,\dots,|v''|\right\}$,
\begin{equation}
 \|x'_n - x_n\| \le \|x'_0-x_0\|= |w'|.\label{eq:nonexprec}
\end{equation}
Now observe that by construction, $x_{|v''|}=\left[w''\right]$ which, together with (\ref{eq:nonexprec}), implies that
\begin{equation*}
|w| = \left\|x'_{|v''|}\right\|\le \left\|x'_{|v''|} - x_{|v''|}\right\|+\left\|x_{|v''|}\right\|\le |w'|+|w''|.
        \end{equation*}
\end{proof}

From Propositions \ref{prop:nonexp1}, \ref{prop:nonexp2} and \ref{prop:nonexpo}, we deduce the following, 
\begin{corollary}
\label{cor:sub1}
The matching policies Random (including Priorities and {\sc u}) and {\sc ml} are sub-additive.
\end{corollary}

\subsection{{\sc fcfm} and {\sc lcfm}}
\label{subsec:proofFIFOLIFO}
As Example \ref{ex:fcfmlcfm} demonstrates, we cannot exploit a non-expansiveness property similar to 
(\ref{eq:defnonexp1}) for the disciplines {\sc fcfm} and {\sc lcfm}. 
\begin{ex}\label{ex:fcfmlcfm} 
Consider the graph of Figure \ref{Fig:paw}. 
Then, regardless of $\sigma$ we have for instance that
\begin{align*}
\left\|\left[133 \odot_{\textsc{fcfm}} (2,\sigma)\right] - \left[311 \odot_{\textsc{fcfm}}(2,\sigma)\right]\right\| 
&=\left\|\left[33\right] - \left[11\right]\right\|=4 > 2 = \left\|\left[133\right] - \left[311\right]\right\|;\\
\left\|\left[331\odot_{\textsc{lcfm}}(2,\sigma)\right] - \left[113 \odot_{\textsc{lcfm}}(2,\sigma)\right]\right\| 
&=\left\|\left[33\right] - \left[11\right]\right\|= 4 > 2 =\left\|\left[331\right] - \left[11\right]\right\|.
\end{align*}
\end{ex}

We nevertheless have the following result,

\begin{proposition}
\label{prop:sub2}
The matching policies \textsc{fcfm} and \textsc{lcfm} are sub-additive. 
\end{proposition}

\begin{proof}
As we cannot apply the arguments of Proposition \ref{prop:nonexpo},  
we resort to a direct proof for both {\sc fcfm} (related to the proof of Lemma 4 in \cite{ABMW17}) and {\sc lcfm}. 
We keep the notation of Definition \ref{def:subadd}, where we drop for short the dependence 
on $\varsigma$ in the notations $\Mfcfs(.)$ and $\Mlcfs(.)$, as the various {\sc fcfm} and {\sc lcfm} matchings do not 
depend on any list of preferences. 

\paragraph{FCFM.} 
We start with the policy \textsc{fcfm}. We proceed in two steps, 

 {\bf Step I:} Let $|z'|=1$, and assume that $\Mfcfs(z'')$ has $K$ unmatched items. 
We need to show that $\Mfcfs(z'z'')$ has at
most $K+1$ unmatched items. 
There are three possible cases:

(a) The item $z'$ is unmatched in $\Mfcfs(z'z'')$ . 
Then, by the definition of {\sc fcfm} $z' \pv z''_j$ for any letter $z''_j$ of $z''$. 
Again from the definition of {\sc fcfm}, the presence in line of this incompatible item $z'$ does not influence the choice of 
match of any subsequent item of the word $z''$. Thus the matched pairs in $\Mfcfs(z'z'')$ are exactly the ones in $\Mfcfs(z'')$, so there are $K+1$ unmatched items in $\Mfcfs(z'z'')$.

(b) The item $z'$ gets matched in $\Mfcfs(z'z'')$ with an unmatched item $z''_{j_1}$ of $\Mfcfs(z'')$. Then, any unmatched item in $\Mfcfs(z'')$ remains unmatched in $\Mfcfs(z'z'')$. 
On another hand, for any matched item $z''_i$ in $\Mfcfs(z'')$ (let $z''_j$ be its match), either $z''_i \pv z''_{j_1}$, and thus choses its match in $\Mfcfs(v)$ regardless of whether $z''_{j_1}$ is matched or not, 
and thus choses again $z''_j$, or $z''_i \v z''_{j_1}$ and thus from the {\sc fcfm} property, we have $j < j_1$ and in turn $z''_j$ remains matched with $z''_j$ in $\Mfcfs(z'z'')$. Therefore the matching induced by the letters of $z''$ in $\Mfcfs(z'z'')$ remains precisely $\Mfcfs(z'')$, so 
$\Mfcfs(z'z'')$ has $K-1$ unmatched items.

(c) The item $z'_1$ gets matched with an item $z''_{j_1}$ that was matched in $\Mfcfs(z'')$ to some item $z''_{i_1}$. The {\sc fcfm} matching of $z'_{1}$ with $z''_{j_1}$ 
breaks the old match $(z''_{i_1}, z''_{j_1})$, so we now need to search a new match for $z''_{i_1}$.
Either there is no {\sc fcfm} match for $z''_{i_1}$ and we stop, or we find a match $z''_{j_2}$.
The new pair $(z''_{i_1}, z''_{j_2})$ potentially broke an old pair $(z''_{i_2}, z''_{j_2})$. 
We continue on and on, until either $z''_{i_k}$ cannot find a new match or $z''_{j_k}$ was not previously matched, and consequently, 
with $K$ unmatched items in the first case and $K-1$ in the second. 
Observe that due to the {\sc fcfm} property, we have $i_{\ell} \leq i_{\ell +1}$ and $j_{\ell} \leq j_{\ell +1}$ for all $\ell\le k$.

{\bf Step II:} Consider now an arbitrary finite word $z'$. Observe, that if $(z'_i,z'_j) \in \Mfcfs(z')$, then
$(z'_i,z'_j) \in \Mfcfs(z'z'')$, as is the case for any admissible policy. Thus, denoting $w'= W_{\textsc{fcfm}}(z')$, we have $W_{\textsc{fcfm}}(z'z'')=W_{\textsc{fcfm}}\left(w'z''\right)$. 
Denote $w' = w'_1 \ldots w'_p$.  
We will consider one by one the items in $w'$, from right to left. If we denote for all $1 \leq i \leq p$, 
$(\Mfcfs)_i= \Mfcfs(w'_{p-i+1}\ldots w'_pz'')$ and $K_i$, the number of unmatched items in $(\Mfcfs)_i$, Step I entails by an immediate induction that for all 
$1 \leq i \leq p$, $K_{i} \leq i + \left|W_\phi(z'')\right|.$ Hence we 
finally have \[\left|W_\phi(z'z'')\right|=K_p \le p + \left|W_\phi(z'')\right| = \left|W_\phi(z')\right|+\left|W_\phi(z'')\right|.\] 

\paragraph{LCFM.} 
We now turn to {\sc lcfm}, for which we apply the same procedure as above, 

{\bf Step I:} Set $|z'|=1$ and assume that $\Mfcfs(z'')$ has $K$ unmatched items. The three different cases are the same as above, 

(a) If $z'_1$ is unmatched in $\Mlcfs(z'z'')$, then $z'_1$ is incompatible with $z''_1$, otherwise the two items would have been matched. 
In turn, if follows from the definition of {\sc lcfm} that the presence in line of $z'_1$ does not influence the choice of match of any item 
$z''_j$ that is matched in $\Mlcfs(z'')$, even though $z'_1\v z''_j$. So $\Mlcfs\left(z'z''\right)$ has exactly $K+1$ unmatched items.  

(b) Whenever $z'_1$ is matched in $\Mlcfs(z'z'')$ with an item $z''_{j_1}$ that was unmatched in $\Mlcfs(z''),$ any 
matched item $z''_i$ in $\Mlcfs(z'')$ that is compatible with $z''_{j_1}$ has found in $z''$ a more recent compatible match $z''_j$. 
The matching of $z''_i$ with $z''_j$ still occurs in $\Mlcfs(z'z'')$. 
Thus, as above the matching induced in $\Mlcfs(z'z'')$ by the nodes of $z''$ is not affected by the 
match $(z'_1,z''_{j_1})$, so there are are $K-1$ unmatched items in $\Mlcfs(z'z'')$.  

(c) Suppose now that $z'_1$ is matched with a server $z''_{j_1}$ that is matched in $\Mlcfs(z'')$. 
We proceed as for {\sc fcfm}, by constructing the new corresponding matchings $\left(z'_1,z''_{j_1}\right)$, $\left(z''_{i_1},z''_{j_2}\right)$, $\left(z''_{i_2},z''_{j_3}\right)$, 
and so on, until we reach the same conclusion as for {\sc fcfm} (with the only difference that in {\sc lcfm} the indexes $i_1,i_2,...$ and $j_1,j_2,...$ are not necessarily ordered increasingly).   
Therefore, at Step I we reach the same conclusions as for {\sc fcfm}.

{\bf Step II:} The construction for {\sc fcfm} remains valid for any admissible policy, such as {\sc lcfm}.   
\end{proof}

\begin{ex}[Example \ref{ex:MS} continued: {\sc ms} is not even sub-additive]
\label{ex:MSbis}
We saw in example \ref{ex:MS} that the matching policy {\sc ms} is not non-expansive. As a matter of fact, it is not even sub-additive. 
Indeed, take again the graph of Figure \ref{Fig:paw} as a compatibility graph. Letting $z'=11$ and $z''=133224$, we immediately obtain that  
for all $\zeta'$, $\zeta''$, $\left|W_{\textsc{ms}}(z'z'',\varsigma'\varsigma'')\right| =4$ and $\left|W_{\textsc{ms}}(z',\varsigma')\right|+\left|W_{\textsc{ms}}(z'',\varsigma'')\right|=2,$ see Figure \ref{fig:example2bis}. 
\end{ex}

\begin{figure}[h!]
\begin{center}
\begin{tikzpicture}[scale=0.7]
\draw[-] (-1,1) -- (8,1);
\fill (0,1) circle (2pt) node[below] {\small{1}};
\fill (1,1) circle (2pt) node[below] {\small{1}};
\draw[-,very thick] (1.45,0.6) -- (1.45,1.4);
\draw[-,very thick] (1.55,0.6) -- (1.55,1.4);
\fill (2,1) circle (2pt) node[below] {\small{1}};
\fill (3,1) circle (2pt) node[below] {\small{3}};
\fill (4,1) circle (2pt) node[below] {\small{3}};
\draw[-, thin] (4,1) .. controls +(up:0.5cm)  .. (7,1);
\fill (5,1) circle (2pt) node[below] {\small{2}};
\draw[-, thin] (2,1) .. controls +(up:0.5cm)  .. (5,1);
\fill (6,1) circle (2pt) node[below] {\small{2}};
\draw[-, thin] (3,1) .. controls +(up:0.5cm)  .. (6,1);
\fill (7,1) circle (2pt) node[below] {\small{4}};
\draw[-] (-1,-0.5) -- (8,-0.5);
\fill (0,-0.5) circle (2pt) node[below] {\small{1}};
\fill (1,-0.5) circle (2pt) node[below] {\small{1}};
\fill (2,-0.5) circle (2pt) node[below] {\small{1}};
\fill (3,-0.5) circle (2pt) node[below] {\small{3}};
\fill (4,-0.5) circle (2pt) node[below] {\small{3}};
\draw[-, thin] (4,-0.5) .. controls +(up:0.5cm)  .. (6,-0.5);
\fill (5,-0.5) circle (2pt) node[below] {\small{2}};
\draw[-, thin] (3,-0.5) .. controls +(up:0.5cm)  .. (5,-0.5);
\fill (6,-0.5) circle (2pt) node[below] {\small{2}};
\fill (7,-0.5) circle (2pt) node[below] {\small{4}};
\end{tikzpicture}
\caption[smallcaption]{'Match the Shortest' is not sub-additive.} 
\label{fig:example2bis}
\end{center}
\end{figure}
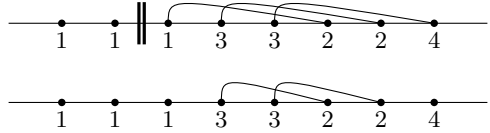

\section{Coupling from the past}
\label{sec:coupling}
The stationary state of matching models is in general, not known in closed form. 
The only remarkable exception is the case where $\phi$ is `First Come, First Matched', for which it is shown that the stationary distribution can be given in closed form, see Theorem 1 in \cite{MBM21}. 
But, as observed in a similar context in Section 5.4 of \cite{ABMW17}, the computation of the normalizing constant in the latter form can be intractable for compatibility graphs that have more than a few nodes and not many edges. Hence the need for alternative techniques to compute, or at least simulate, the stationary state of the system. 
As is well known, strong backwards coupling convergence (in the sense - specified below - of Borovkov and Foss) can lead to a perfect simulation algorithm of the equilibrium, by sampling values of the Markov chain under consideration (here, the buffer-content process $\suite{W_n}$) at the coalescence time, see Section \ref{sec:PW}. 
This Section is devoted to the construction of the steady state of the system, using strong backwards coupling convergence. 
As will be seen below, these coupling results will also guarantee in many cases, the existence of a unique stationary buffer content, and thereby, of a unique bi-infinite stationary complete matching, in a sense that will be specified in Section \ref{sec:matching}. 

The argument will be more easily developed in the ergodic-theoretical framework that we introduce in Sub-section \ref{subsec:statergo}. Our main coupling result, Theorem \ref{thm:main}, is stated in Sub-section \ref{subsec:main}. 
Subsections \ref{subsec:erase} (with the introduction of the useful concept of {\em erasing words}) and \ref{subsec:proofmain} are then devoted to the proof of Theorem \ref{thm:main}.

\subsection{Framework}
\label{subsec:statergo}
The general matching model is intrinsically periodic: arrivals are simple but departure are pairwise, so the size of the system has the parity of the original size at all even times - in particular a system started empty can possibly be empty only every other time, and two systems cannot couple unless their initial sizes have almost surely the same parity. 
To circumvent this difficulty, at first we track the system only at even times. Equivalently, we change the time scale and see the arrivals by {\em pairs} of items (as in the original bipartite matching model \cite{CKW09,BGMa12,ABMW17}) which play different roles: the first one investigates first all possible matchings in the buffer upon its arrival according to $\phi$, before possibly considering the second one if no match is available, whereas the second one applies $\phi$ to all available items, including the first one. By doing so, it is immediate to observe that we obtain exactly a GM model as presented thus far, only 
do we track it at even times. 
  
To formalize the above observation, we let $\suite{U_n}$ be the buffer content sequence at even times (we will use the term "{\em even} buffer content"), that is, $U_n=W_{2n}$, $n\in\N$. 
We will primarily construct a stationary version of the sequence $\suite{U_n}$ of even buffer contents, by coupling. 
For this, we work on the canonical space $\Omega^0:=\left(\maV\times\mathcal S\times\maV\times\mathcal S\right)^\mathbb Z$ 
of the bi-infinite sequence $\suitez{\left(V_{2_n},\Sigma_{2n},V_{2n+1},\Sigma_{2n+1}\right)}$, 
on which we define the bijective shift operator $\theta$ by $\theta\left((\omega_n)_{n\in\mathbb Z}\right)= (\omega_{n+1})_{n\in\mathbb Z}$ for all $(\omega_n)_{n\in \mathbb Z} \in \Omega$. 
We denote by $\theta^{-1}$ the reciprocal operator of $\theta$, and by $\theta^n$ and $\theta^{-n}$ the $n$-th iterated of 
$\theta$ and $\theta^{-1}$, respectively, for all $n\in\N$. We equip $\Omega^0$ with a sigma-algebra $\mathscr F^0$ and with the image probability measure 
$\bp$ of the sequence $\suitez{\left(V_{2_n},\Sigma_{2n},V_{2n+1},\Sigma_{2n+1}\right)}$ on $\Omega^0$. 
Observe that under the IID assumptions for the input, $\bp$ is compatible with the shift, i.e. 
for any $\maA \in \mathscr F^0$, $\bpr{\maA}=\bpr{\theta^{-1}\maA}$ and any $\theta$-invariant event $\maB$ ({\em i.e.} such that $\maB=\theta^{-1}\maB$) is either $\bp$-negligible or almost sure.  
Altogether, the quadruple $\mathscr Q^0:=\left(\Omega^0,\mathscr F^0,\bp,\theta\right)$ is thus stationary ergodic, and will be referred to as {\em Canonical space} of the input at even times. 
For more details about this framework, we refer the reader to the monographs \cite{BranFranLis90},
\cite{BacBre02} (Sections 2.1 and 2.5) and \cite{Rob03} (Chapter 7). 

We define the r.v. $\left(V^0,\Sigma^0,V^1,\Sigma^1\right)$ by $\left(V^0,\Sigma^0,V^1,\Sigma^1\right)(\omega_n)=\left(v_{0},s_{0},v_{1},s_{1}\right)$, 
for any $\suitez{\omega_n}:=\suitez{\left(v_{2_n},s_{2n},v_{2n+1},s_{2n+1}\right)}$ in $\Omega^0$.  
Thus $\left(V^0,\Sigma^0,V^1,\Sigma^1\right)$ can be interpreted as the input brought to the system at time 0, i.e. at 0 an item of class $V^0$ and then an item of class $V^1$ enter the system, having respective lists of preferences $\Sigma^0$ and $\Sigma^1$ over 
$\mathcal V$, and the order of arrival between the two is kept track of ($V^0$ and {\em then} $V^1$). 
Then for any $n\in \mathbb Z$, the r.v. $\left(V^0\circ\theta^n,\,\Sigma^0\circ\theta^n,\,V^1\circ\theta^n,\,\Sigma^1\circ\theta^n\right)$ corresponds to the input brought to the system at time $n$. 
Define the following subsets of $\mathbb W$, 
\begin{align}
\mathbb W_2 &= \left\{w \in \mathbb W\,:\,|w|\mbox{ is even }\right\};\nonumber\\
\mathbb W_2(r) &= \left\{w \in \mathbb W_2\,:\,|w|\le 2r\right\},\,r \in \N_+. \label{eq:defW2r}
\end{align}
For any $\mathbb W_2$-valued r.v. $Y$, we define on $\Omega^0$ the sequence $\suite{U^{\{Y\}}_n}$ as the even buffer content sequence of the model initiated at value $Y$, i.e.  
\begin{equation}
\label{eq:recurUbar}
\left\{\begin{array}{ll}
U^{\{Y\}}_0 &= Y;\\
U^{\{Y\}}_{n+1} &= \left(U^{\{Y\}}_{n} \odot_\phi (V^0\circ\theta^n,\Sigma^0\circ\theta^n)\right)\odot_\phi (V^1\circ\theta^n,\Sigma^1\circ\theta^n),\,n\in\N.
\end{array}\right.
\end{equation} 
A stationary version of (\ref{eq:recurUbar}) is thus a recursion satisfying (\ref{eq:recurUbar}) and compatible with the shift, i.e. 
a sequence $\suitez{U\circ\theta^n}$, where the $\mathbb W_2$-valued r.v. $U$ satisfies the functional equation  
\begin{equation}
\label{eq:recurstatUbar}
U\circ\theta = \left(U \odot_\phi (V^0,\Sigma^0)\right) \odot_\phi (V^1,\Sigma^1),
\end{equation}
see Section 2.1 of \cite{BacBre02} for details. 
To any stationary even buffer content $\suitez{U\circ\theta^n}$ corresponds a unique stationary probability for the sequence $\suite{W_{2n}}$ on the original probability space 
$(\Omega,\mathcal F,\mathbb P)$. 
Moreover, as will be shown in Section \ref{sec:matching}, provided that $\bpr{U=\emptyset}>0$ the bi-infinite sequence $\suitez{U\circ\theta^n}$ corresponds on $\mathscr Q^0$ to a unique stationary matching by $\phi$ (we write a $\phi$-{\em matching}), 
that is obtained by using the bi-infinite family of construction points $\left\{n \in \Z\,:\,U \circ\theta^n =\emptyset\right\}$, and matching the incoming items 
by $\phi$, within each finite block between construction points.

\subsection{Main result}
\label{subsec:main}
It follows from the discussion above that the construction of a stationary buffer-content at even times and thereby, of a stationary $\phi$-matching on $\mathbb Z$, amounts to solving on $\mathscr Q^0$ the almost-sure equation (\ref{eq:recurstatUbar}). This will be done by constructing the associated {backwards scheme}, as in \cite{Loynes62}:  for a $\mathbb W_2$-valued r.v. $Y$ and any fixed $n \ge 0$, the r.v. $U^{\{Y\}}_n\circ\theta^{-n}$ represents the even buffer content at time 0, whenever initiated at value $Y$, $n$ time epochs in the past. 
Loynes' theorem \cite{Loynes62} shows the existence of a solution to (\ref{eq:recurstatUbar}), as the $\bp$-almost sure limit of the non-decreasing sequence $\suite{U^{\{\emptyset\}}_n\circ\theta^{-n}}$, whenever 
the random map driving the recursion $\suite{U_n}$ is almost surely non-decreasing in the state variable. In the absence of a clear monotonicity in the recursion dynamics, we resort instead to Borovkov's and Foss theory of Renovation, see \cite{Foss92,Foss94}. 

Following \cite{Bor84}, we say that the buffer content sequence $\suite{U^{\{Y\}}_n}$ converges with {\em strong backwards coupling} to the stationary buffer content sequence 
$\suite{U\circ\theta^n}$ if, $\bp$-almost surely there exists a finite $N^*\ge 0$ such that for all $n \ge N^*$, $U^{\{Y\}}_{n}=U$. Note that strong backwards coupling implies the (forward) coupling between $\suite{U^{\{Y\}}_n}$ and $\suite{U\circ\theta^n}$, i.e.  
there exists a.s. an integer $N\ge 0$ such that $U^{\{Y\}}_{n}=U\circ\theta^n$ for all $n \ge N$. In particular the distribution of $U^{\{Y\}}_{n}$ converges in total variation to that of $U$, see e.g. Section 2.4 of \cite{BacBre02}. 

\medskip

Denote for any $\mathbb W_2$-valued r.v. $Y$ and any $j\in\mathbb N^*$, by $\tau_j(Y)$ the $j$-th visit time to $\emptystate$ (or return time if $Y\equiv \emptyset$) 
for the process $\suite{U_n^{\{Y\}}}$, that is 
$$\tau_1(Y) := \inf \left\{n > 0, U_{n}^{\{Y\}} = \emptystate \right\}, \quad \tau_j(Y) := \inf \left\{n > \tau_{j-1}(Y), U_{n}^{\{Y\}} = \emptystate \right\}, \; j\geq 2.$$ 
The stability of the system is characterized by the following condition depending on the initial condition $Y$,  
\begin{itemize}
\item[\textbf{(H1)}] The stopping time $\tau_1(Y)$ is integrable.
\end{itemize}
We are ready to state our main coupling result. 
\begin{theorem}
\label{thm:main}
If the policy $\phi$ is sub-additive and assumption (H1) holds, there exists a unique solution $U$ to (\ref{eq:recurstatUbar}) in $\mathscr Y_2^{\infty}$, to which all sequences 
$\suite{U^{\{Y\}}_n}$, for $Y \in \mathscr Y_2^{\infty}$, converge with strong backwards coupling. 
\end{theorem} 
At this point, it is useful to provide a list of simple cases in which Theorem \ref{thm:main} applies. 
It is well know that the following set of measures plays a key role in the stability of the system at hand 
(see e.g. the survey in Section 3 of \cite{BMM23}): For any compatibility graph $G=(\maV,\maE)$,  
\begin{equation}
\label{eq:Ncond}
\textsc{Ncond}(G)=\left\{\mu \in \mathscr M(\maV)\,:\,\mu(\maI) < \mu(\maE(\maI)),\,\mbox{ for all independent sets }\maI \mbox{ of }G\right\}.
\end{equation}
Observe that the Markov chain $\suite{U_n}$ is clearly irreducible on $\mathbb W_2$. 
So (H1) holds true whenever the chain is positive recurrent. Therefore, applying Theorem \ref{thm:main} of \cite{MBM21} for {\sc fcfm}, and Theorem 2 of \cite{MaiMoy16}, we obtain the following list of sufficient conditions for (H1),

\begin{proposition}
\label{prop:suffH3}
Condition (H1) holds true for any $\mathbb W_2$-valued initial condition $Y$, 
whenever $G$ is non-bipartite, $\mu\in\textsc{Ncond}(G)$, and in either one of the following cases:  
\begin{enumerate}
\item $\phi=\textsc{fcfm}$; 
\item $\phi=\textsc{ml}$; 
\item $\phi$ is any admissible policy and $G$ is complete $p$-partite for $p\ge 3$. 
\end{enumerate}
\end{proposition}


Theorem \ref{thm:main} is proved in Sub-section \ref{subsec:proofmain}. For this, we first need to introduce the notion of erasing words. 

\subsection{(Strong) Erasing words}
\label{subsec:erase}
\begin{definition}
Let $G=(\maV,\maE)$ be a connected graph, and $\phi$ be an admissible matching policy. Let $u \in \mathbb W_2$. 
We say that the word $z\in \maV^*$ is an {\em erasing word} of $u$ for $(G,\phi)$ if $|z|$ is even and 
for any two words $\varsigma'$ and $\varsigma$ possibly drawn by $\nu_\phi$ on $\mathcal S^*$ and  having respectively the same size as $z$ and $u$, we have that 
\begin{equation}
\label{eq:deferase}
W_\phi\left(z,\varsigma'\right)=\emptyset\quad \quad \mbox{ and }\quad\quad W_\phi\left(uz,\varsigma\varsigma'\right)=\emptyset.
\end{equation}
\end{definition}
In other words, an erasing word of $u$ has the twofold property of being perfectly matchable by $\phi$ alone, and together with $u$. 
The following proposition guarantees the existence of erasing words for any stabilizable graph and any sub-additive policy. 
\begin{proposition}
\label{pro:erasing}
Let $G$ be a non-bipartite graph and $\phi$ be a sub-additive matching policy. Then any word $u\in \mathbb W_2$ admits an erasing 
word for $(G,\phi)$. 
\end{proposition}
Proposition \ref{pro:erasing} is proven in Section \ref{sec:appendixA}.

Clearly, uniqueness of the erasing words does not hold true. In particular, if $z^1$ and $z^2$ are both erasing words of the same word $u$ for $(G,\phi)$, then $z^1z^2$ also is. Hence the following, 
\begin{definition}
Let $u \in \mathbb W_2$. An erasing word $z$ of $u$ for $(G,\phi)$ is said to be {\em reduced}, if $z$ cannot be written as $z=z^1z^2$, where $z^1$ and $z^2$ are both non-empty erasing words of $u$.  A reduced erasing word $z$ of $u$ is said to be {\em minimal}, if it is of minimal length among all reduced erasing words of $u$. 
\end{definition}

\medskip 

To show our perfect simulation result, we will also need to strengthen the concept of erasing word. 
\begin{definition}
\label{def:strongerase}
Let $C\in\N_+$. A word $z \in \maV^*$ of even length $2p$ is said to be a $2C$-{\em strong erasing word} of the graph $G=(\maV,\maE)$ and the matching policy $\phi$ if 
\begin{enumerate}
\item[(i)] $z$ is completely matchable by $\phi$ together with any word of $\W_{2}(C)$, i.e. for any $w\in\W_{2}(C)$ and any two words $\varsigma$ and $\varsigma'$ of $\mathcal S^*$ whose letters can be possibly drawn by $\nu_\phi$ and of respective length $|w|$ and $2p$, we have that $W_\phi\left(wz,\varsigma\varsigma'\right)=\emptyset$; 
\item[(ii)] for any even prefix $\breve z$ of $z$ of length $2r$, for any $w\in\W_{2}(C)$ and any two words $\varsigma$ and $\breve\varsigma'$ of $\mathcal S^*$ whose letters can be possibly drawn by $\nu_\phi$ and of respective length $|w|$ and $2r$, we have that 
$W_\phi\left(w\breve z,\varsigma\breve\varsigma'\right) \le W_\phi\left(w,\varsigma\right)$. 
\end{enumerate}
\end{definition}
In other words, (i) a $2C$-strong erasing word of $(G,\phi)$ is perfectly matchable by $\phi$ alone, and together with any buffer content of size less or equal than $2C$ and (ii), the corresponding input never leads to an increase of the buffer size before depleting the system. In particular, plainly, a $2C$-strong erasing word for $(G,\phi)$ is a an erasing word for any $w\in \mathbb W_2(C)$. 

\begin{definition}
An $2q$-strong erasing word $z$ for $(G,\phi)$ is said to be reduced, if $z$ cannot be written as $z=z^1z^2$, where $z^1$ and $z^2$ are both non-empty $2C$-strong erasing words.  
\end{definition}

We start by observing the following composition rule regarding sub-additive policies, 
\begin{lemma}
\label{lemma:reducestrong}
Let $\phi$ be a sub-additive matching policy on $G$. Then, for any $C\in\N_+$, and any family 
$z^1,\cdots,z^C$ of (possibly equal) $2$-strong erasing words of $(G,\phi)$, the word $z=z^1z^2\cdots z^C$ is a $2C$-strong erasing word of $(G,\phi)$.
\end{lemma}
\begin{proof}
The arguments of this proof do not depend on the drawn lists of preferences, as long as they are fixed upon arrival. Again, we thus skip this parameter from all notations. It is immediate that assertion (ii) of Definition \ref{def:strongerase} holds for all $C$. We now show that it also the case for (i), and for this we proceed by induction on $C$. 
The property (i) holds by definition for $C=1$. Now suppose that it holds for a given $C$. Then, take a word $w\in\W_{2}(C+1)$ and a family $z^1,\cdots,z^{C+1}$ of $2$-strong erasing words. 
If $w\in\W_{2}(C)$, then $W_\phi(wz^1z^2\cdots z^{C})=\emptyset$ by the induction assumption, and thus 
\[W_\phi\left(wz^1\cdots z^{C}z^{C+1}\right)=W_\phi\left(W_\phi(wz^1\cdots z^{C})z^{C+1}\right)
=W_\phi\left(\emptyset z^{C+1}\right)=\emptyset.\]
Now suppose that $|w|=2(C+1)$. We write $w=w^1w^2\cdots w^{C+1}$, where the $w^k$'s are two-letter words of the form $w^k=ij$ for $i\pv j$. Again, we have that 
\begin{equation}
\label{eq:obel}
W_\phi\left(wz^1\cdots z^{C+1}\right) = W_\phi\left(W_\phi(wz^1\cdots z^{C})z^{C+1}\right)
=W_\phi\left(W_\phi\left(w^1w^2\cdots w^{C+1}z^1\cdots z^{C}\right)z^{C+1}\right).
\end{equation}
But, by the sub-additivity property and the recurrence assumption we get that 
\[|W_\phi\left(w^1w^2\cdots w^{C+1}z^1\cdots z^{C}\right)| \le |W_\phi(w^1)|+|W_\phi\left(w^2\cdots w^{C+1}z^1\cdots z^{C}\right)|=|W_\phi(w^1)|=2,\]
so $W_\phi\left(w^1w^2\cdots w^{C+1}z^1\cdots z^{C}\right)$ is either a two-letter word of the form $ij$ for $i\pv j$, or the empty word. Injecting this in \eqref{eq:obel}, in the first case we obtain that 
\[W_\phi\left(wz^1\cdots z^{C+1}\right) 
=W_\phi\left(ijz^{C+1}\right)=\emptyset,\]
while in the second, we get that 
\[W_\phi\left(wz^1\cdots z^{C+1}\right) 
=W_\phi\left(z^{C+1}\right)=\emptyset,\]
because $z^{C+1}$ is a $2$-strong erasing word. Hence $z^1\cdots z^{C+1}$ is a $2(C+1)$-strong erasing word, which concludes the proof.  

\end{proof}


%
%

 Hereafter we provide a list of cases for which the existence of $2C$-strong erasing words is granted, 

\begin{proposition}
\label{pro:erasing2}
Let $C\in\N_+$. 
The following conditions are sufficient for the existence of a $2C$-strong erasing word for $(G,\phi)$: 
\begin{itemize}
\item[(i)] $G$ is complete $p$-partite for $p\ge 3$ and $\phi$ is any sub-additive policy; 
\item[(ii)] $G$ is non-bipartite, and $\phi=\textsc{lcfm}$;
\item[(iii)] $G$ is an odd cycle, and $\phi=\textsc{fcfm}$;
\item[(iv)] $G$ is the `paw' graph of Figure \ref{Fig:paw}, and $\phi$ is any sub-additive policy. 
\end{itemize}
\end{proposition}

\begin{proof}
In view of Proposition \ref{lemma:reducestrong}, as $\phi$ is sub-additive it is enough to check that there exists a $2$-strong erasing word $z^1$ in all cases, since for all $C$, it is then enough to set $z=\underbrace{z^1\cdots z^1}_{C}$ to obtain a $2C$-strong erasing word. 

(i) Suppose that $G$ is complete $p$-partite for $p \ge 3$, and let $\maI_1,...,\maI_p$ be the corresponding maximal independent sets. 
Let $z$ be a word of length $2p$ containing exactly two (possibly identical) letters belonging to each one of the $\maI_i$'s, $i=1,...,p$, but such each letter of odd index and the immediate succeeding letter are of two different maximal independent sets 
$\maI_j$ and $\maI_k$, $j \ne k$. Then it is immediate that $z^1$ is a $2$-strong erasing word for $(G,\phi)$ for any admissible $\phi$. 

(ii) The proof of Assertion (ii) is given in Appendix \ref{sec:appendixB}. 
 
(iii) The proof of Assertion (iii) is given in Appendix \ref{sec:appendixC}. 

(iv) It can be immediately checked by hand that $z^1=234234$ is a $2$-strong erasing word for the paw graph of Figure \ref{Fig:paw}. 
\end{proof}

\begin{remark}\rm
It is immediate that the result of (ii) can be extended to any admissible policy $\phi$ inducing the same choices as {\sc lcfm} on the input $ijz^1$, for any $i\pv j$. This is true in particular for any priority policy such that, in the spanning cycle $\mathscr C$, for any $j\in \llbracket 2,2q+1\rrbracket$, $c_j$ prioritizes $c_{j-1}$ over any other node, and $c_1$ prioritizes $c_{2q+1}$ over any other node. 
\end{remark}
\begin{ex}
\rm
\label{ex:emullcfmsep}
Consider the compatibility graph $G$ on $\maV=\llbracket 1,6 \rrbracket$, represented in Figure \ref{Fig:sep2}. It is immediate that $G$ is complete $3$-partite, of $3$-partition 
$$\maI_1\cup\maI_2\cup\maI_3:=\{1,4\}\cup\{2,5\}\cup\{3,6\}.$$
Then from the proof of (i) above, the word $z^1=121653$ is a strong erasing word for any admissible $\phi$. 
\begin{figure}[h!]
\begin{center}
\begin{tikzpicture}[scale=0.8]
\draw[-] (1,1) -- (2,-0.2);
\draw[-] (1,1) -- (3,0);
\draw[-] (1,1) -- (2,1.2);
\draw[-] (1,1) -- (3,1);
\draw[-] (2,1.2) -- (1,0);
\draw[-] (2,1.2) -- (3,0);
\draw[-] (2,1.2) -- (3,1);
\draw[-] (3,1) -- (1,0);
\draw[-] (3,1) -- (2,-0.2);
\draw[-] (1,0) -- (2,-0.2);
\draw[-] (1,0) -- (3,0);
\draw[-] (2,-0.2) -- (3,0);
\fill (1,1) circle (2pt) node[above] {\small{1}} ;
\fill (2,1.2) circle (2pt) node[above right] {\small{2}} ;
\fill (3,1) circle (2pt) node[above] {\small{3}} ;
\fill (1,0) circle (2pt) node[below] {\small{4}} ;
\fill (2,-0.2) circle (2pt) node[below] {\small{5}} ;
\fill (3,0) circle (2pt) node[below] {\small{6}} ;
\end{tikzpicture}
\caption[smallcaption]{$3$-partite complete compatibility graph.}
\label{Fig:sep2}
\end{center}
\end{figure}
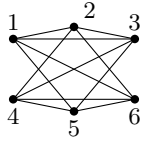
\end{ex}


We believe that the above list of sufficient conditions for the existence of strong erasing is far from exhaustive. 
Checking that given words are strongly erasing for particular graph geometries is a very interesting problem 
involving intricate combinatorial arguments - see the proofs of assertions (ii) and (iii) in Appendix \ref{sec:appendixB} and \ref{sec:appendixC}. We leave this more thorough investigation for future research.

\subsection{Proof of Theorem \ref{thm:main}}
\label{subsec:proofmain} 
Define the following family of events for any $\mathbb W_2$-valued r.v. $Y$,  
\begin{align*}
\maA_n(Y) &=\left\{U^{\{Y\}}_{n}=\emptyset\right\}= \Bigl\{W_\phi(YV^0V^1V^0\circ\theta\,V^1\circ\theta\,...\,V^0\circ\theta^{n-1}V^1\circ\theta^{n-1})=\emptyset\Bigl\},\quad n\ge 0, 
\end{align*}
We first have the following result, 
\begin{proposition}
\label{pro:bwiid}
For any $Y \in \mathscr Y_2^{\infty}$, if (H1) holds, then we have 
\begin{equation}
\label{eq:renov0}
\lim_{n\to\infty} \bpr{\bigcap_{k=0}^{\infty} \bigcup_{l=0}^n \maA_{l}(Y) \cap \theta^k\maA_{l+k}(Y)}=1.
\end{equation}
\end{proposition}
\begin{proof}
Fix throughout $\epsilon>0$, and $r\in \N_+$ such that $Y \in\mathscr Y_2^r$. 
As a consequence of the integrability of $\tau_1(Y)$, the random variable 
\[\kappa = \sup\left\{k\in \N\,:\,\tau_1(Y)\circ\theta^{-k}>k\right\}\]
that is, the largest horizon in the past from which the first visit to $\emptyset$ takes place after time 0, is a.s. finite. In particular there exists a positive integer $K_{\epsilon}$ such that 
\begin{equation}
\label{eq:prK}
\bpr{\kappa > K_\epsilon} < {\epsilon \over 5}.
\end{equation}
Again in view of (H3), there exists an integer $T_\epsilon>0$ such that 
\begin{equation}
\label{eq:prT}
\bpr{ \tau_1(Y)> T_\epsilon} < {\epsilon \over 5},
\end{equation}
and let us denote $H_\epsilon:=2K_\epsilon + 2r + T_\epsilon.$ 
We know from Proposition \ref{pro:erasing} that any word admits at least one minimal erasing word. Also, there are finitely many words in $\mathbb W_2$ of size less than $H_\epsilon$, 
and thus finitely many minimal erasing words of those words. So the following integer is well defined, and depends only on $H_\epsilon$, 
\begin{align}
\ell_\epsilon &= {1\over 2}\max_{u \in \mathbb W_2: |w| \le H_\epsilon} \min_{\substack{z\in \maV^*:\\ z\mbox{ \scriptsize{minimal erasing word of }}u}} |z|.\label{eq:defell}
\end{align}

We now define the sequence $\suitei{\tilde\tau_i}$ (where we drop the dependence on $Y$ for notational convenience), as the following subsequence of $\suitei{\tau_i(Y)}$: 
$$\tilde\tau_1 :=\tau_1(Y), \quad \tilde\tau_i := \inf \left\{n > \tilde\tau_{i-1} + \ell_\epsilon,\, U_{n}^{\{Y\}} = \emptystate \right\}, \; i\geq 2.$$ 
Also define the following family of events: for all $k\in\N$ and $i\in\N_+$,  
\begin{equation*}
\maD^{k}_i(Y) =\bigcup_{m=1}^{\ell_\epsilon} \Bigl\{V^0\circ\theta^{\tilde\tau_i + k}\,V^1\circ\theta^{\tilde\tau_i +k}\,...\,V^0\circ\theta^{\tilde\tau_i+k+m-1}\,V^1\circ\theta^{\tilde\tau_i+k+m-1}
\mbox{ \small{is an erasing word of} }U^{\{Y\}}_{\tilde\tau_i+k}\Bigl\},
\end{equation*}
and for any $k,n\in\N$, 
\begin{equation}
\maD^{k,n}(Y) =\bigcup_{\substack{i\in\N_+:\\\tilde\tau_i + \ell_\epsilon \le 2n}} \,\,\maD^{k}_i(Y),\quad k\in\N,\,n \in \N_+. \label{eq:defDkn}
\end{equation}
For any $k \in \N$ and $i\in\N_+$, on $\theta^k\maD^{k}_i(Y)$ we first have that for some (unique, and even) integer $m \le \ell_\epsilon$, 
$U^{\{Y\}}_{\tilde\tau_i+k+m}\circ\theta^{-k}= \emptyset$, and second, that $U^{\{Y\}}_{\tilde\tau_i+m}=\emptyset$, since 
\begin{multline*}
\left|U^{\{Y\}}_{\tilde\tau_i+m}\right|=\left|W_\phi\left(Y\,V^0\,V^1\,...\,V^0\circ\theta^{\tilde\tau_i+m-1}\,V^1\circ\theta^{\tilde\tau_i+m-1}\right)\right|\\
\shoveleft{\le \left|W_\phi\left(Y\,V^0\,V^1\,...\,V^0\circ\theta^{\tilde\tau_i-1}\,V^1\circ\theta^{\tilde\tau_i-1}\right)\right|}\\
\shoveright{+\left|W_\phi\left(V^0\circ\theta^{\tilde\tau_i}\,V^1\circ\theta^{\tilde\tau_i}\,...\,V^0\circ\theta^{\tilde\tau_i+m-1}\,V^1\circ\theta^{\tilde\tau_i+m-1}\right)\right|}\\ 
=0,
\end{multline*}
where the two terms in the third and fourth line above are zero from the very definitions of $\tilde\tau_i$ and an erasing word, respectively. 
Consequently, we have that 
\begin{equation}
\label{eq:finale1}
\theta^k\maD^{-k,n}(Y) \subseteq \bigcup_{l=0}^n \maA_{l}(Y) \cap \theta^k\maA_{l+k}(Y),\quad k,n \in \N.  
\end{equation}
Second, fix $n\in\N$ and a sample $\omega \in \{\kappa \le K_\epsilon\}\cap\bigcap_{k'=0}^{K_\epsilon}\theta^{k'}\maD^{k',n}(\emptyset)$ and an integer 
$k \ge K+1$. By the definition of $\kappa$, $U_0\left(\theta^{-k}\omega\right)=Y(\theta^{-k}\omega)$ entails that $U_{k-k'}\left(\theta^{-k}\omega\right)=\emptyset$ for some $k'\le K_\epsilon$;   
in other words $U^{\{\emptyset\}}_n \left(\theta^{-k'}\omega\right)$ equals $U^{\{Y\}}_{n+k-k'}\left(\theta^{-k}\omega\right)$ for any $n\ge 0$. But as $\theta^{-k'}\omega \in \maD^{k',n}(\emptyset)$ by assumption, 
we obtain that $\theta^{-k}\omega \in \maD^{k,n}(Y)$. Consequently we have that 
\begin{equation*}
\{\kappa \le K_\epsilon\}\cap\bigcap_{k=0}^{K_\epsilon}\theta^{k'}\maD^{k',n}(\emptyset) \subseteq \{\kappa \le K_\epsilon\}\cap\bigcap_{k=K+1}^{\infty}\theta^k\maD^{k,n}(Y)
\end{equation*}
and thereby
\[\{\kappa \le K_\epsilon\}\cap\bigcap_{k=0}^{K_\epsilon}\theta^k\left(\maD^{k,n}(Y) \cap \maD^{k,n}(\emptyset)\right) \subseteq \{\kappa \le K_\epsilon\}\cap\bigcap_{k=0}^{\infty}\theta^k\maD^{k,n}(Y).\]
This, together with (\ref{eq:finale1}), yields for any $n\in\N$ to 
\begin{equation}
\label{eq:finale2}
\{\kappa \le K_\epsilon\}\cap\bigcap_{k=0}^{K_\epsilon}\theta^k\left(\maD^{k,n}(Y) \cap \maD^{k,n}(\emptyset)\right) \subseteq \{\kappa \le K_\epsilon\}\cap\bigcap_{k=0}^{\infty} \bigcup_{l=0}^n \maA_{l}(Y) \cap \theta^k\maA_{l+k}(Y).  
\end{equation}

Now recall (\ref{eq:defell}). In words, $\ell_\epsilon$ is (half of) the minimal length of word that can accommodate at least one erasing word of any admissible word of even size 
bounded by $H_\epsilon$. Therefore, in view of the iid assumptions the following is a well defined element of $]0,1[$: 
\begin{multline}
\label{eq:defbeta}
\beta_\epsilon = \min_{u \in \mathbb W_2\,:\,|u| \le H_\epsilon} \bp\Biggl[\bigcup_{m=1}^{\ell_\epsilon}\Bigl\{V^0\,V^1\,V^0\circ\theta\,V^1\circ\theta\,...\,V^0\circ\theta^{m-1}\,V^0\circ\theta^{m-1}\\
 \mbox{ is a minimal erasing word of } u\Bigl\}\Biggl]. 
\end{multline}
Let 
\[M_\epsilon=\left\lceil {\mbox{Log}\epsilon - \mbox{Log}5 - \mbox{Log} (K_\epsilon+1) \over \mbox{Log}(1-\beta_\epsilon)}\right\rceil,\] 
that is, the least integer that is such that 
\begin{equation}
\label{eq:prM}
\left(1-\beta_\epsilon\right)^{M_\epsilon} < {\epsilon \over 5(K_\epsilon+1)}. 
\end{equation}
Again from (H3) and (IID), there exists a positive integer $N_\epsilon$ such that 
\begin{equation}
\label{eq:prN}
\bpr{\tilde \tau_{M_\epsilon} +  \ell_\epsilon > N_\epsilon}<{\epsilon \over 5}.
\end{equation}
All in all, we obtain that for all $n > N_\epsilon$, 
\begin{multline}
\bpr{\overline{\bigcap_{k=0}^{\infty} \bigcup_{l=0}^n \maA_{l}(Y) \cap \theta^k\maA_{l+k}(Y)}}\\
\shoveleft{\le \bpr{\overline{\bigcap_{k=0}^{\infty} \bigcup_{l=0}^n \maA_{l}(Y) \cap \theta^k\maA_{l+k}(Y)}\cap \{\tilde \tau_{M_\epsilon} +  \ell_\epsilon \le N_\epsilon\} 
\cap \{\kappa \le K_\epsilon\} \cap \{\tau_1(Y) \le  T_\epsilon\}}}\\ \shoveright{+\bpr{\tilde \tau_{M_\epsilon} +  \ell_\epsilon > N_\epsilon} + \bpr{\kappa > K_\epsilon}+ \bpr{\tau_1(Y)> T_\epsilon}}\\
\shoveleft{\le \bpr{\overline{\bigcap_{k=0}^{K_\epsilon}\theta^k\left(\maD^{k,n}(Y) \cap \maD^{k,n}(\emptyset)\right)} \,\,\cap \{\tilde \tau_{M_\epsilon} +  \ell_\epsilon \le N_\epsilon\} \cap \{\tau_1(Y) \le  T_\epsilon\}}  
+ {3\epsilon \over 5}}\\
\le \sum_{k=0}^{K_\epsilon}\bpr{\left(\bigcap_{i=1}^{M_\epsilon}\,\overline{\theta^k\maD^{k}_i(Y)}\right)\cap \{\tau_1(Y) \le  T_\epsilon\}}
+\sum_{k=0}^{K_\epsilon}\bpr{\left(\bigcap_{i=1}^{M_\epsilon}\,\overline{\theta^k\maD^{k}_i(\emptyset)}\right)\cap \{\tau_1(Y) \le  T_\epsilon\}}  + {3\epsilon \over 5}, \label{eq:finale3}
\end{multline}
where we use (\ref{eq:prK}), (\ref{eq:prT}), (\ref{eq:finale2}) and (\ref{eq:prN}) in the second inequality, and recalling (\ref{eq:defDkn}). 
Now let $u_\epsilon$ be an element of $\mathbb W_2$ such that $\left|u_\epsilon\right| \le H_\epsilon$ and achieving the minimum in (\ref{eq:defbeta}), i.e. 
\[\beta_\epsilon = \bpr{\bigcup_{m=1}^{\ell_\epsilon}\Bigl\{\mbox{$V^0\,V^1\,V^0\circ\theta\,V^1\circ\theta\,...\,V^0\circ\theta^{m-1}\,V^0\circ\theta^{m-1}$ \small{is a minimal erasing word of} $u_\epsilon$}\Bigl\}},\]
and define the events 
\begin{align*}
\check{\maD}_i &= \bigcup_{m=1}^{\ell_\epsilon} \Bigl\{\mbox{$V^0\circ\theta^{\tilde\tau_i}\,V^1\circ\theta^{\tilde\tau_i}\,...\,V^0\circ\theta^{\tilde\tau_i+m-1}\,V^0\circ\theta^{\tilde\tau_i+m-1}$ \small{is a minimal erasing word of }$u_\epsilon$}\Bigl\},\, i\in \N.
\end{align*}
From assumption (IID), the events $\check{\maD}_i,i\in\N$, are iid of probability $\beta_\epsilon$. 
On another hand, on the event $\{\tau_1(Y) \le  T_\epsilon\}$, for any $0 \le k \le K_\epsilon$, 
\[\left|U^{\{Y\}}_{\tau_1(Y)+k}\circ\theta^{-k}\right| \le |Y| + 2k + \tau_1(Y)  \le 2r + 2K_\epsilon +T_\epsilon = H_\epsilon.\]
 Thus, as $W_\phi\left(V^0\circ\theta^{\tilde\tau_i}\,V^1\circ\theta^{\tilde\tau_i}\,...\,V^0\circ\theta^{\tilde\tau_{i+1}-1}\,V^1\circ\theta^{\tilde\tau_{i+1}-1}\right)=\emptyset$ 
for all $i$, the sub-additivity of $\phi$ and an immediate induction entail that $\left|U^{\{Y\}}_{\tilde \tau_i + k}\circ\theta^{-k}\right| \le H_\epsilon$ for all $i \ge 1$. 
Therefore, for any $k \le K_\epsilon$ and any $i\in \N_+$, by the very definition of $\beta_\epsilon$ we have that 
$\bpr{\theta^k\maD^{k}_i(Y)} \ge \bpr{\check{\maD}_i}=\beta_\epsilon$, and in turn by independence of the $\check{\maD}_i$'s, that for all $k \le K_\epsilon$, 
\begin{equation}
\label{eq:finale4}
\bpr{\left(\bigcap_{i=1}^{M_\epsilon}\,\overline{\theta^k\maD^{k}_i(Y)}\right)\cap \{\tau_1(Y) \le  T_\epsilon\}} \le \prod_{i=1}^{M_\epsilon}\,\bpr{\overline{\check{\maD}_i}} = \left(1-\beta_\epsilon\right)^{M_\epsilon}.
\end{equation} 
All the same, on the event $\{\tau_1(Y) \le  T_\epsilon\}$, for any $0 \le k \le K_\epsilon$ we have that  
\[\left|U^{\{\emptyset\}}_{\tau_1(Y)+k}\circ\theta^{-k}\right| \le 2k + \tau_1(Y)  \le H_\epsilon,\]
thus we can conclude similarly that 
\[\bpr{\left(\bigcap_{i=1}^{M_\epsilon}\,\overline{\theta^k\maD^{k}_i(\emptyset)}\right)\cap \{\tau_1(Y) \le  T_\epsilon\}} \le \left(1-\beta_\epsilon\right)^{M_\epsilon}.\]
Injecting this together with (\ref{eq:finale4}) and (\ref{eq:prM}) in (\ref{eq:finale3}) entails that, for any $n > N_\epsilon$, 
\[\bpr{\overline{\bigcap_{k=0}^{\infty} \bigcup_{l=0}^n \maA_{l}(Y) \cap \theta^k\maA_{l+k}(Y)}} < \epsilon,\]
which concludes the proof.  
\end{proof} 

We now prove the following forward coupling result,  

\begin{proposition}
\label{pro:fwiid}
Under (H1), there is forward coupling between $\suite{U^{\{Y\}}_n}$ and $\suite{U^{\{Y^*\}}_n}$, for any 
two r.v.'s $Y$ and $Y^*$ in $\mathscr Y_2^\infty$.  
\end{proposition}

\begin{proof}
We aim at proving that the stopping time 
\[\rho(Y,Y^*):= \inf\left\{n \ge 0: U^{\{Y\}}_l = U^{\{Y^*\}}_l\mbox{ for all }l \ge n\right\}\]
is a.s. finite, that is 
\begin{equation}
\label{eq:fwcouple1}
\lim_{n\to\infty}\bpr{\rho(Y,Y^*) \le n}=1.
\end{equation}
Observe that, as the two recursions $\suite{U^{\{Y\}}_n}$ and $\suite{U^{\{Y^*\}}_n}$ are driven by the same input, they coalesce as soon as they meet for the first time. Hence, 
(\ref{eq:fwcouple1}) holds true in particular if 
\begin{equation}
\label{eq:fwcouple2}
\lim_{n\to\infty}\bpr{\bigcup_{l=0}^n \,\Bigl\{U^{\{Y\}}_l=U^{\{Y^*\}}_l=\emptyset\Bigl\}}=\lim_{n\to\infty}\bpr{\bigcup_{l=0}^n \,\mathscr A_l(Y)\cap \mathscr A_l(Y^*)}=1.
\end{equation} 
From Proposition \ref{pro:bwiid}, the latter holds true whenever we replace $Y^*$ by $U_0^{\{Z\}}\circ\theta^{-k}$ for any finite $\mathbb W_2$-valued r.v. $Z$ and any $k \in \N$. 
The proof of (\ref{eq:fwcouple2}) for any finite $Y^*$ is analog. 
\end{proof}

We are now ready to conclude the proof of Theorem \ref{thm:main}. 

\begin{proof}[Proof of Theorem \ref{thm:main}]
Fix a r.v. $Y\in\mathscr Y_2^{\infty}$, and let $r$ be such that $Y \in \mathscr Y_2^r$. From Proposition \ref{pro:bwiid}, (\ref{eq:renov0}) holds true. Moreover, the sequence 
$\suite{\maA_n(Y)}$ is clearly a sequence of renovating events of length 1 for 
$\suite{U^{\{Y\}}_{n}}$ (see \cite{Foss92,Foss94}). It then follows from Theorem 2.5.3 of \cite{BacBre02} that $Y$ converges with strong backwards coupling, and thereby also in the forward sense, to a stationary sequence $\suite{U\circ\theta^n}$, where $U\in\mathscr Y_2^{\infty}$. Now, Proposition \ref{pro:fwiid} entails in particular that any pair of such stationary sequences $\suite{U\circ\theta^n}$ and $\suite{U^*\circ\theta^n}$ couple, and therefore coincide almost surely. Thus there exists a unique solution $U$ to (\ref{eq:recurstatUbar}) in $\mathscr Y_2^{\infty}$.   
\end{proof}

\subsection{Consequences}
\begin{corollary}
\label{cor:uniqueiid}
Under the assumptions of Theorem \ref{thm:main}, the bi-infinite stationary version of the even buffer content sequence 
$\suitez{U_n}$ is unique. 
\end{corollary} 

\begin{proof}
Consider a couple of stationary versions $\suitez{U^*_n}$ and $\suitez{\tilde U^*_n}$. Fix $m\in \Z$. 
Then we have by definition that 
\[\{U^*_m \ne \tilde U^*_m\} = \bigcap_{n\le m} \left\{U^{\{U^*_{n}\}}_m\circ\theta^n \ne U^{\{\tilde U^*_{n}\}}_m\circ\theta^n \right\}
:=  \bigcap_{n\le m} \mathcal B_n.\]
The sequence of events $\left(\mathcal B_n\right)_{n\le m}$ is clearly decreasing for inclusion as $n$ decreases - it is in fact constant. 
Therefore we get that 
\begin{equation}
\label{eq:unique1}
\pr{U^*_m \ne \tilde U^*_m} = \pr{\bigcap_{n\le m} \mathcal B_n} = \lim_{n\to -\infty} \pr{\mathcal B_n}.
\end{equation}
Now, denoting for any $n \in \Z$, 
\[N^+(n) = \inf\left\{k\ge n\,:\,U^{\{U^*_n\}}_k\circ\theta^n = U^{\{\tilde U^*_n\}}_k\circ\theta^n\right\} =\inf\left\{k\ge n\,:\,U^*_k = \tilde U^*_k\right\},\]
we get that for any $n\le m$, 
\[\pr{\mathcal B_n} = \pr{N^+(n) > m - n} = \pr{N^+(0) > m - n},\]
where we use the stationarity of the input $\suitez{(V^0_n,\Sigma^0_n,V^1_n,\Sigma^1_n)}$, 
and of $\suitez{U^*_n}$ and $\suitez{\tilde U^*_n}$ in the second inequality. 
But applying Proposition \ref{pro:fwiid} to the r.v.'s $U^*_0$ and $\tilde U^*_0$, we obtain that $N^+(0)$ is a.s. finite. 
Thus 
$$\lim_{n\to -\infty} \pr{N^+(0) > m - n} = 0,$$
and we conclude using (\ref{eq:unique1}) that $U^*_m = \tilde U^*_m$ a.s.. As this is true for any $m\in \Z$, 
the proof is complete. 

\end{proof}

We deduce from the above results that  

\begin{corollary}
\label{cor:bwiid}
Under the assumptions of Theorem \ref{thm:main}, for any $Y \in \mathscr Y_2^{\infty}$ and any $m\in\Z$ we have that 
\[\lim_{k\to\infty} U^{\{Y\}}_m\circ\theta^{-k} = U^*_m \quad \mbox{ a.s.,}\]
where $\suitez{U^*_n}$ is the only stationary version of $\suitez{U_n}$. 
\end{corollary}

\begin{proof}
In view of Theorem \ref{thm:main}, there exists a.s. a finite $N^-$ such that  $U^{\{Y\}}_0\circ\theta^{-k} = U^*_0$ 
for all $k \ge N$. Hence the claim for $m=0$. 
It can be generalized to any fixed $m\in\Z$ in view of the stationarity of the input and of $\suitez{U^*_n}$. 
%
%
\end{proof}

\section{$\phi$-matchings}
\label{sec:matching} 


As is easily seen, any model $(G,\mu,\phi)$ generates a family of random graphs, as follows. 
For any $n\in \N$, we consider the matching  
$$\mathbf M^{\{\emptyset,0\}}_{n}(\phi):=M_\phi(V^0_0V^1_{0}\,...\,V^0_{n-1}V^0_{n-1}\,,\,\Sigma^0_0\Sigma^1_0\,...\,\Sigma^0_{n-1}\Sigma^1_{n-1}),$$ 
that is, the random graph in which the nodes are the entered items from $0$ to $n-1$ (on the even time scale introduced above),  
and there is an edge between two nodes if and only if the corresponding items are matched according to $\phi$. 
 In any realization of $\mathbf M^{\{\emptyset,0\}}_{n}(\phi)$, all nodes have thus 0 or 1 neighbor, in other words they are 
of degree 0 or 1. We can naturally extend this definition by denoting, for a $\mathbb W_2$-valued r.v. $Y$, by 
$\mathbf M^{\{Y,0\}}_{n}(\phi)$ the matching of all initially stored items (represented by the initial condition $Y$), together with the items entered 
up to time $n$ excluded. The realization of a finite matching $\mathbf M^{\{Y,0\}}_{n}(\phi)$ is then said to be {\em perfect} 
if all of its nodes are of degree 1. 

\subsection{Infinite $\phi$-matchings at even times}  
It is immediate that for any $Y \in \mathscr Y_2^{\infty}$ and any $n$, 
$\mathbf M^{\{Y,0\}}_{n}(\phi)$ is a.s. an induced subgraph of 
$\mathbf M^{\{Y,0\}}_{n+1}(\phi)$, see an example for $\phi=\textsc{lcfm}$ in Figure \ref{fig:increase}. 

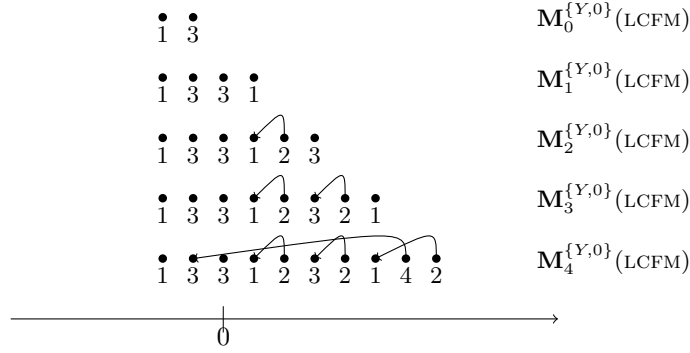
\begin{figure}[h!]
\begin{center}
\begin{tikzpicture}[scale=0.8]
\fill (-1,5) circle (2pt) node[below] {\small{1}};
\fill (-0.5,5) circle (2pt) node[below] {\small{3}};
\fill (5,5) node[right] {\small{$\mathbf M^{\{Y,0\}}_0(\textsc{lcfm})$}};
\fill (-1,4) circle (2pt) node[below] {\small{1}};
\fill (-0.5,4) circle (2pt) node[below] {\small{3}};
\fill (0,4) circle (2pt) node[below] {\small{3}};
\fill (0.5,4) circle (2pt) node[below] {\small{1}};
\fill (5,4) node[right] {\small{$\mathbf M^{\{Y,0\}}_1(\textsc{lcfm})$}};
\fill (-1,3) circle (2pt) node[below] {\small{1}};
\fill (-0.5,3) circle (2pt) node[below] {\small{3}};
\fill (0,3) circle (2pt) node[below] {\small{3}};
\fill (0.5,3) circle (2pt) node[below] {\small{1}};
\fill (1,3) circle (2pt) node[below] {\small{2}};
\draw[->, thin] (1,3) .. controls +(up:0.5cm)  .. (0.5,3);
\fill (1.5,3) circle (2pt) node[below] {\small{3}};
\fill (5,3) node[right] {\small{$\mathbf M^{\{Y,0\}}_2(\textsc{lcfm})$}};
\fill (-1,2) circle (2pt) node[below] {\small{1}};
\fill (-0.5,2) circle (2pt) node[below] {\small{3}};
\fill (0,2) circle (2pt) node[below] {\small{3}};
\fill (0.5,2) circle (2pt) node[below] {\small{1}};
\fill (1,2) circle (2pt) node[below] {\small{2}};
\draw[->, thin] (1,2) .. controls +(up:0.5cm)  .. (0.5,2);
\fill (1.5,2) circle (2pt) node[below] {\small{3}};
\fill (2,2) circle (2pt) node[below] {\small{2}};
\draw[->, thin] (2,2) .. controls +(up:0.5cm)  .. (1.5,2);
\fill (2.5,2) circle (2pt) node[below] {\small{1}};
\fill (5,2) node[right] {\small{$\mathbf M^{\{Y,0\}}_3(\textsc{lcfm})$}};
\fill (-1,1) circle (2pt) node[below] {\small{1}};
\fill (-0.5,1) circle (2pt) node[below] {\small{3}};
\fill (0,1) circle (2pt) node[below] {\small{3}};
\fill (0.5,1) circle (2pt) node[below] {\small{1}};
\fill (1,1) circle (2pt) node[below] {\small{2}};
\draw[->, thin] (1,1) .. controls +(up:0.5cm)  .. (0.5,1);
\fill (1.5,1) circle (2pt) node[below] {\small{3}};
\fill (2,1) circle (2pt) node[below] {\small{2}};
\draw[->, thin] (2,1) .. controls +(up:0.5cm)  .. (1.5,1);
\fill (2.5,1) circle (2pt) node[below] {\small{1}};
\fill (3,1) circle (2pt) node[below] {\small{4}};
\draw[->, thin] (3,1) .. controls +(up:0.5cm)  .. (-0.5,1);
\fill (3.5,1) circle (2pt) node[below] {\small{2}};
\draw[->, thin] (3.5,1) .. controls +(up:0.5cm)  .. (2.5,1);
\fill (5,1) node[right] {\small{$\mathbf M^{\{Y,0\}}_4(\textsc{lcfm})$}};
\draw[->] (-3.5,0) -- (5.5,0);
\fill (0,0) node[]{$|$} node[below]{$0$};
\end{tikzpicture}
\caption[smallcaption]{Construction of the increasing sequence 
$\suite{\mathbf M^{\{Y,0\}}_n(\phi)}$ for a given sample $\omega$, $Y(\omega)=13$, 
$\phi=\textsc{lcfm}$ and $G$ the compatibility graph of 
Figure \ref{Fig:paw}.}
\label{fig:increase}
\end{center}
\end{figure} 

We aim at constructing from the increasing sequence $\suite{\mathbf M^{\{Y,0\}}_{n}(\phi)}$, the limiting object 
$\mathbf M^{\{Y,0\}}_{\infty}(\phi)$ 
as the infinite random graph obtained when letting $n$ go to infinity in the above. 
For this, suppose that $G$ is non-bipartite, $\phi$ is sub-additive and $\mu\in\textsc{Ncond}(G)$. Consider the (unique, from Theorem  \ref{thm:main}) stationary version $\suitez{U^*_n}$ 
of the even buffer content chain of the model. 
In these conditions, $U^*_n \sim \Pi_U$ for all $n$ and $\Pi_U(\emptyset)>0$, so the family of integers 
\[\mathscr C^*_{\ge 0}:=\left\{n \in \N\,:\, U^*_n = \emptyset\right\}\]
is a.s. infinite. The elements of the latter are called {\em construction points} of the model over $\N$. 
In particular, an infinite $\phi$-matching $\mathbf M^{*,0}_{\infty}(\phi)$ can be constructed from the stationary version $\suitez{U^*_n}$ 
as the union of the $\phi$-matchings between construction points. In other words, letting 
$c^*_0 < c^*_1 < c^*_2 <... $ be the elements of $\mathscr C^*_{\ge 0}$ in increasing order, we set 
(with obvious notation) 
\begin{equation}
\label{eq:definfmatch}
\mathbf M^{*,0}_{\infty}(\phi):=\mathbf M^{\{U^*_0,0\}}_{\infty}(\phi)=\mathbf M^{\{U^*_0,0\}}_{c^*_{0}}(\phi)\cup\,\bigcup_{i=0}^{\infty} \mathbf M^{\{\emptyset,c^*_i\}}_{c^*_{i+1}}(\phi),
\end{equation} 
that is, the random graph obtained by concatenating the perfect matching of the initially stored items corresponding to $U^*_0$ together with all arrivals 
until the first non-negative construction point, and all other perfect matchings between construction points. 
Thus according to the above definition, the infinite matching $\mathbf M^{*,0}_{\infty}(\phi)$ is perfect. We call it the {\em stationary} perfect infinite $\phi$-matching (at even times) 
of the model. 

Now, in view of Theorem \ref{thm:main} there is strong backwards coupling for the Markov chain $\suite{U^{\{Y\}}_n}$ 
with the stationary version $\suitez{U^*_n}$.  
In particular, there is also forward coupling between these two sequences, and we let $N^+$ be the a.s. finite coupling time. 
Letting also
$J:=\inf\{i\in\N\,:\,c^*_i \ge N^+\}$ be the first construction point after the coupling time, 
denote 
$$\mathbf M^{\{Y,0\}}_{\infty}(\phi):=\mathbf M^{\{Y,0\}}_{c^*_{J}}(\phi)\cup\,\bigcup_{i=J}^{\infty} \mathbf  M^{\{\emptyset,c^*_i\}}_{c^*_{i+1}}(\phi).$$
Because $U^{\{Y,0\}}_{c^*_{J}}=\emptyset$, the infinite $\phi$-matching $\mathbf M^{\{Y,0\}}_{\infty}(\phi)$ is perfect, and coincides with $\mathbf M^{*,0}_{\infty}(\phi)$ from $c^*_J$ onwards. 
We have thus proven the following result, 

\begin{proposition}
\label{prop:matchingN}
Under the assumptions of Theorem \ref{thm:main}, for any $Y \in \mathscr Y_2^{\infty}$ 
there exists a.s. a unique infinite $\phi$-matching at even times, 
$\mathbf M^{\{Y,0\}}_{\infty}(\phi)$, starting with the even buffer-content $Y$ at time $0$. 
This infinite matching is a.s. perfect, and coincide a.s. in finite time with the stationary infinite 
$\phi$-matching $\mathbf M^{*,0}_{\infty}(\phi)$ of the model, defined by (\ref{eq:definfmatch}). 
\end{proposition}

\begin{remark}
\rm
In graph theory, a matching on a graph is an induced subgraph in which any node is of {degree} 0 or 1. 
A matching is said {maximal} if there exists no matching strictly inducing it, i.e., there is no matching including it and having more edges. It is said {perfect} if it contains only nodes of degree 1. 
There is a simple and insightful connection between $\phi$--matchings and matchings on infinite random graphs. In fact, the present procedure  
builds the random graph {\em together} with a matching on it, and the matching algorithm is `online', in the sense that once an edge is added to the matching it cannot be discarded latter to optimize the matching size. 
See details in section 9 of \cite{MoyPer17}, or in \cite{SJM23}. 
\end{remark}

\subsection{Bi-infinite $\phi$-matchings at even times}

Clearly, by the stationarity assumption, Proposition \ref{prop:matchingN} can be generalized to any arbitrary starting point $m\in\Z$ instead of time 0, and we replace for any associated variable, the superscript $^{.,0}$ by $^{.,m}$. 
One is then naturally led to consider {\em bi-infinite} $\phi$-matchings, by also letting $m$ go to $-\infty$ in this construction. 
However, doing so requires more care. Indeed, adding an arrival at the beginning of an input may change the whole matching of that input. 
Specifically, for any $m\in\Z$, it is not true in general that $\mathbf M^{\{Y,m\}}_{\infty}(\phi)$ is an induced subgraph of $\mathbf M^{\{Y,m-1\}}_{\infty}(\phi)$, 
see an example for $\phi=\textsc{lcfm}$ on Figure \ref{fig:nonincrease}. 

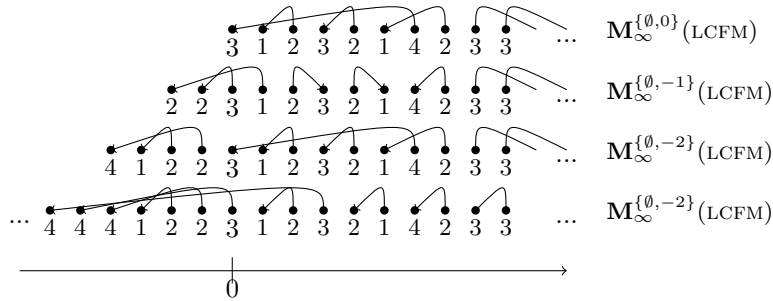
\begin{figure}[h!]
\begin{center}
\begin{tikzpicture}[scale=0.8]
\fill (0,4) circle (2pt) node[below] {${3}$};
\fill (0.5,4) circle (2pt) node[below] {\small{1}};
\fill (1,4) circle (2pt) node[below] {\small{2}};
\draw[->, thin] (1,4) .. controls +(up:0.5cm)  .. (0.5,4);
\fill (1.5,4) circle (2pt) node[below] {\small{3}};
\fill (2,4) circle (2pt) node[below] {\small{2}};
\draw[->, thin] (2,4) .. controls +(up:0.5cm)  .. (1.5,4);
\fill (2.5,4) circle (2pt) node[below] {\small{1}};
\fill (3,4) circle (2pt) node[below] {\small{4}};
\draw[->, thin] (3,4) .. controls +(up:0.5cm)  .. (0,4);
\fill (3.5,4) circle (2pt) node[below] {\small{2}};
\draw[->, thin] (3.5,4) .. controls +(up:0.5cm)  .. (2.5,4);
\fill (4,4) circle (2pt) node[below] {\small{3}};
\draw[-, thin] (4,4) .. controls +(up:0.5cm)  .. (5,4);
\fill (4.5,4) circle (2pt) node[below] {\small{3}};
\draw[-, thin] (4.5,4) .. controls +(up:0.5cm)  .. (5.5,4);
\fill (5.5,4) node[below] {...};
\fill (6,4) node[right] {\small{$\mathbf M^{\{\emptyset,0\}}_{\infty}(\textsc{lcfm})$}};
\fill (-1,3) circle (2pt) node[below] {\small{2}};
\fill (-0.5,3) circle (2pt) node[below] {\small{2}};
\fill (0,3) circle (2pt) node[below] {${3}$};
\draw[->, thin] (0,3) .. controls +(up:0.5cm)  .. (-0.5,3);
\fill (0.5,3) circle (2pt) node[below] {\small{1}};
\draw[->, thin] (0.5,3) .. controls +(up:0.5cm)  .. (-1,3);
\fill (1,3) circle (2pt) node[below] {\small{2}};
\fill (1.5,3) circle (2pt) node[below] {\small{3}};
\draw[->, thin] (1,3) .. controls +(up:0.5cm)  .. (1.5,3);
\fill (2,3) circle (2pt) node[below] {\small{2}};
\fill (2.5,3) circle (2pt) node[below] {\small{1}};
\draw[->, thin] (2,3) .. controls +(up:0.5cm)  .. (2.5,3);
\fill (3,3) circle (2pt) node[below] {\small{4}};
\fill (3.5,3) circle (2pt) node[below] {\small{2}};
\draw[->, thin] (3.5,3) .. controls +(up:0.5cm)  .. (3,3);
\fill (4,3) circle (2pt) node[below] {\small{3}};
\draw[-, thin] (4,3) .. controls +(up:0.5cm)  .. (5,3);
\fill (4.5,3) circle (2pt) node[below] {\small{3}};
\draw[-, thin] (4.5,3) .. controls +(up:0.5cm)  .. (5.5,3);
\fill (5.5,3) node[below] {...};
\fill (6,3) node[right] {\small{$\mathbf M^{\{\emptyset,-1\}}_{\infty}(\textsc{lcfm})$}};
\fill (-2,2) circle (2pt) node[below] {\small{4}};
\fill (-1.5,2) circle (2pt) node[below] {\small{1}};
\fill (-1,2) circle (2pt) node[below] {\small{2}};
\draw[->, thin] (-1,2) .. controls +(up:0.5cm)  .. (-1.5,2);
\fill (-0.5,2) circle (2pt) node[below] {\small{2}};
\draw[->, thin] (-0.5,2) .. controls +(up:0.5cm)  .. (-2,2);
\fill (0,2) circle (2pt) node[below] {${3}$};
\fill (0.5,2) circle (2pt) node[below] {\small{1}};
\fill (1,2) circle (2pt) node[below] {\small{2}};
\draw[->, thin] (1,2) .. controls +(up:0.5cm)  .. (0.5,2);
\fill (1.5,2) circle (2pt) node[below] {\small{3}};
\fill (2,2) circle (2pt) node[below] {\small{2}};
\draw[->, thin] (2,2) .. controls +(up:0.5cm)  .. (1.5,2);
\fill (2.5,2) circle (2pt) node[below] {\small{1}};
\fill (3,2) circle (2pt) node[below] {\small{4}};
\draw[->, thin] (3,2) .. controls +(up:0.5cm)  .. (0,2);
\fill (3.5,2) circle (2pt) node[below] {\small{2}};
\draw[->, thin] (3.5,2) .. controls +(up:0.5cm)  .. (2.5,2);
\fill (4,2) circle (2pt) node[below] {\small{3}};
\draw[-, thin] (4,2) .. controls +(up:0.5cm)  .. (5,2);
\fill (4.5,2) circle (2pt) node[below] {\small{3}};
\draw[-, thin] (4.5,2) .. controls +(up:0.5cm)  .. (5.5,2);
\fill (5.5,2) node[below] {...};
\fill (6,2) node[right] {\small{$\mathbf M^{\{\emptyset,-2\}}_{\infty}(\textsc{lcfm})$}};
\fill (-3,1) circle (2pt) node[below] {\small{4}};
\fill (-2.5,1) circle (2pt) node[below] {\small{4}};
\fill (-2,1) circle (2pt) node[below] {\small{4}};
\fill (-1.5,1) circle (2pt) node[below] {\small{1}};
\fill (-1,1) circle (2pt) node[below] {\small{2}};
\draw[->, thin] (-1,1) .. controls +(up:0.5cm)  .. (-1.5,1);
\fill (-0.5,1) circle (2pt) node[below] {\small{2}};
\draw[->, thin] (-0.5,1) .. controls +(up:0.5cm)  .. (-2,1);
\fill (0,1) circle (2pt) node[below] {${3}$};
\draw[->, thin] (0,1) .. controls +(up:0.5cm)  .. (-2.5,1);
\fill (0.5,1) circle (2pt) node[below] {\small{1}};
\fill (1,1) circle (2pt) node[below] {\small{2}};
\draw[->, thin] (1,1) .. controls +(up:0.5cm)  .. (0.5,1);
\fill (1.5,1) circle (2pt) node[below] {\small{3}};
\draw[->, thin] (1.5,1) .. controls +(up:0.5cm)  .. (-3,1);
\fill (2,1) circle (2pt) node[below] {\small{2}};
\fill (2.5,1) circle (2pt) node[below] {\small{1}};
\draw[->, thin] (2.5,1) .. controls +(up:0.5cm)  .. (2,1);
\fill (3,1) circle (2pt) node[below] {\small{4}};
\fill (3.5,1) circle (2pt) node[below] {\small{2}};
\draw[->, thin] (3.5,1) .. controls +(up:0.5cm)  .. (3,1);
\fill (4,1) circle (2pt) node[below] {\small{3}};
\fill (4.5,1) circle (2pt) node[below] {\small{3}};
\draw[-, thin] (4.5,1) .. controls +(up:0.5cm)  .. (4,1);
\fill (5.5,1) node[below] {...};
\fill (6,1) node[right] {\small{$\mathbf M^{\{\emptyset,-2\}}_{\infty}(\textsc{lcfm})$}};
\fill (-3.5,1) node[below] {...};
\draw[->] (-3.5,0) -- (5.5,0);
\fill (0,0) node[]{$|$} node[below]{$0$};
\end{tikzpicture}
\caption[smallcaption]{Backwards construction of the sequence  
$\left(\mathbf M^{\{\emptyset,m\}}_\infty(\phi)\right)_{m\in\Z}$ for $\phi=\textsc{lcfm}$ and $G$, 
the compatibility graph of Figure \ref{Fig:paw}. 
Adding nodes on the left of the matching may break the matches performed at the previous step. }
\label{fig:nonincrease}
\end{center}
\end{figure}

Thus the very definition of what we will call a {\em bi-infinite $\phi$-matching}, i.e., the limiting graph 
obtained when letting $m$ go to $-\infty$ in $\mathbf M^{\{Y,m\}}_{\infty}(\phi)$, is {\em a priori} problematic. 
We show hereafter that this object is well defined, at least under stability and sub-additivity assumptions. 

To see this, first observe that under the assumptions of Theorem \ref{thm:main}, the following set is also a.s. infinite, 
\[\mathscr C^*_{<0}:=\left\{n < 0\,:\, U^*_n = \emptyset\right\},\]
and denote its elements as $... < c^*_{-2} <  c^*_{-1} < 0 $. Altogether, the points $(c^*_i)_{i\in\Z}$ are the construction points of the system on $\Z$. See an example on Figure \ref{fig:construction}. 

\begin{figure}[h!]
\begin{center}
\begin{tikzpicture}[scale=0.8]
\fill (-3.5,1) node[below] {...};
\fill (-3,1) circle (2pt) node[below] {\small{4}};
\fill (-2.5,1) circle (2pt) node[below] {\small{4}};
\fill (-2,1) circle (2pt) node[below] {\small{4}};
\fill (-1.5,1) circle (2pt) node[below] {\small{1}};
\fill (-1,1) circle (2pt) node[below] {\small{2}};
\draw[->, thin] (-1,1) .. controls +(up:0.5cm)  .. (-1.5,1);
\fill (-0.5,1) circle (2pt) node[below] {\small{2}};
\draw[->, thin] (-0.5,1) .. controls +(up:0.5cm)  .. (-2,1);
\fill (0,1) circle (2pt) node[below] {${3}$};
\draw[->, thin] (0,1) .. controls +(up:0.5cm)  .. (-2.5,1);
\fill (0.5,1) circle (2pt) node[below] {\small{1}};
\fill (1,1) circle (2pt) node[below] {\small{2}};
\draw[->, thin] (1,1) .. controls +(up:0.5cm)  .. (0.5,1);
\fill (1.5,1) circle (2pt) node[below] {\small{3}};
\draw[->, thin] (1.5,1) .. controls +(up:0.5cm)  .. (-3,1);
\fill (2,1) circle (2pt) node[below] {\small{2}};
\fill (2.5,1) circle (2pt) node[below] {\small{1}};
\draw[->, thin] (2.5,1) .. controls +(up:0.5cm)  .. (2,1);
\fill (3,1) circle (2pt) node[below] {\small{4}};
\fill (3.5,1) circle (2pt) node[below] {\small{2}};
\draw[->, thin] (3.5,1) .. controls +(up:0.5cm)  .. (3,1);
\fill (4,1) circle (2pt) node[below] {\small{3}};
\fill (4.5,1) circle (2pt) node[below] {\small{3}};
\draw[-, thin] (4.5,1) .. controls +(up:0.5cm)  .. (4,1);
\fill (5.5,1) node[below] {...};
\fill (6,1) node[right] {\phantom{\small{$\mathbf M^{\{\emptyset,-2\}}_{\infty}(\textsc{lcfm})$}}};
\draw[->] (-3.5,0) -- (5.5,0);
\fill (-3,0) node[]{$\star$} node[below]{$c^*_{-1}$};
\fill (0,0) node[]{$|$} node[below]{$0$};
\fill (2,0) node[]{$\star$} node[below]{$c^*_{0}$};
\fill (3,0) node[]{$\star$} node[below]{$c^*_{1}$};
\fill (4,0) node[]{$\star$} node[below]{$c^*_{2}$};
\end{tikzpicture}
\caption[smallcaption]{Construction points of the model for $\phi=\textsc{lcfm}$ and $G$ 
the compatibility graph of Figure \ref{Fig:paw}. }
\label{fig:construction}
\end{center}
\end{figure}
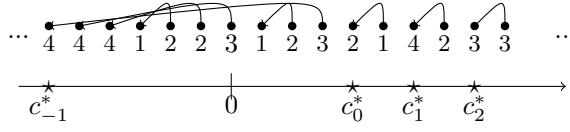 

A unique perfect, bi-infinite, $\phi$-matching can thus be obtained from $\suitez{U^*_n}$. It is formally defined as 
\begin{equation}
\label{eq:defbiinfmatch}
\mathbf M^{*,-\infty}_{\infty}(\phi) = \bigcup_{i\in \Z}\mathbf  M^{\{\emptyset,c^*_i\}}_{c^*_{i+1}}(\phi),
\end{equation} 
and induces $\mathbf M^{*,0}_{\infty}(\phi)$ as a sub-graph. It is called stationary perfect {bi-infinite} $\phi$-matching (at even times) of the model. 
A constructive definition of this object is provided hereafter. 

Observe that or any $m\in\Z$, 
for any $k\in \N$ such that $-k < m$ and any $Y \in \mathscr Y_2^{\infty}$ , the trajectory of the process 
$\suitekk{U^{\{Y\}}_n\circ\theta^{-k}}$ determines uniquely the matches from time $-k$ onwards, and in particular, the restriction of the infinite matching 
$\mathbf M^{\{Y,-k\}}_{\infty}(\phi)$ from time $m$ onwards, that is, $\mathbf M^{\{U^{\{Y,-k\}}_m,m\}}_{\infty}(\phi)$. 
From Corollary \ref{cor:bwiid}, the sequence $\left(U^{\{Y\}}_m\circ\theta^{-k}\right)_{k\in\N}$ is a.s. constant from a certain rank, and 
its limiting value is precisely $U^*_m$. Therefore the following almost sure limit is well defined,  
\[
\lim_{k\to \infty} \mathbf M^{\{U^{\{Y,-k\}}_m,m\}}_{\infty}(\phi) = \mathbf M^{\{U^{*}_m,m\}}_{\infty}(\phi)=
\mathbf M^{*,m}_{\infty}(\phi).\]
As this is true for any $m\in\Z$, we obtain hat the unique stationary bi-infinite matching on $\Z$, namely $\mathbf M^{*,-\infty}_{\infty}(\phi)$ defined by 
(\ref{eq:defbiinfmatch}), a.s. coincides with the above limit. We have proven that 

\begin{theorem}
\label{thm:biinfmatch}
Under the assumptions of Theorem \ref{thm:main},  
there exists a unique perfect bi-infinite $\phi$-matching at even times. 
It coincides almost surely with the 
stationary bi-infinite matching $\mathbf M^{*,-\infty}_{\infty}(\phi)$ defined by (\ref{eq:defbiinfmatch}), and is obtained, for any $Y \in \mathscr Y_2^{\infty}$, by constructing the infinite matching 
$\mathbf M^{\{Y,-k\}}_{\infty}(\phi)$ starting from $Y$ at time $-k$, and letting $k$ go to infinity.  
\end{theorem}

\subsection{Back to the original time scale} 
Come back to the original time scale, i.e., arrivals take place one by one, and we work on the image probability space of the sequence 
$\suitez{(V_n,\Sigma_n)}$, i.e., on the sample space $(\mathcal V\times\mathcal S)^\Z$, equipped with the probability measure 
$\bigotimes_{\Z}(\mu \otimes \nu_\phi)$. Then, first, it is clear how the dynamics of the GM model $(G,\mu,\phi)$ can be used to construct finite, or infinite (on one side) $\phi$-matchings of these items arrived one by one, 
following the same procedure as above. It leads to the analog result as Proposition \ref{prop:matchingN}, only expressed on a different time scale. 

We now turn to the construction of bi-infinite $\phi$-matchings on the original time scale, considering single arrivals. 
Then observe that, due to the 2-periodicity of the model, we have two stationary versions of the natural Markov chain 
$\suitez{W_n}$, depending on whether the item entering at time 0 is of class and lists of preference $(V^0_0,\Sigma^0_0)$ (i.e. it is the first arrival of a batch of two in the even time scale) 
or $(V^1_0,\Sigma^1_0)$ (i.e. it is the second arrival of the batch). 
Specifically, 
we saw that under the assumptions of Theorem \ref{thm:main}, 
there exists a unique stationary even buffer-content sequence $\suitez{U^*_n}$ on the 
canonical space 
of $\suitez{\left(V_{2_n},\Sigma_{2n},V_{2n+1},\Sigma_{2n+1}\right)}$ and likewise, there exists a 
unique stationary even buffer-content sequence $\suitez{U^{**}_n}$ on the canonical space of 
$\suitez{\left(V_{2_n-1},\Sigma_{2n-1},V_{2n},\Sigma_{2n}\right)}$. 
We can then define the two following versions of the recursion $\suitez{W_n}$, 
namely $\suitez{W^*_n}$ and $\suitez{W^{**}_n}$, as follows: 
\begin{align}
\label{eq:defWfinal0}
&\left\{\begin{array}{ll}
W^*_{2n} &=U^*_{n}\\
W^*_{2n+1} &=\left(U^*_{n}\right) \odot_\phi \left(V_{n},\Sigma_{n}\right)
\end{array}
\right.\quad\quad n\in \mathbb Z;\\
\label{eq:defWfinal1}
&\left\{\begin{array}{ll}
W^{**}_{2n} &=\left(U^{**}_{n-1}\right) \odot_\phi \left(V_{2n-1},\Sigma_{2n-1}\right)\\
W^{**}_{2n+1} &=U^{**}_{n},
\end{array}
\right.\quad\quad n\in \mathbb Z,
\end{align}
which correspond respectively to a buffer content sequence that is stationary on $\mathbb W_2$ at even times, and to a buffer 
content sequence that is stationary on $\mathbb W_2$ at odd times. 
By construction, both sequences $\suitez{W^*_n}$ and $\suitez{W^{**}_n}$ have infinitely many construction points (only at even - respectively, odd - times), 
therefore we can construct from each one, a unique bi-infinite perfect $\phi$-matching on the sample space of 
$\suitez{\left(V_{n},\Sigma_{n}\right)}$, having construction points only at even (resp., odd) times. 

Then, following the same steps as in the arguments leading to Theorem 3, these two matchings are explicitly constructed as follows: 
The first matching is obtained 
by constructing (now on the original time scale, with single arrivals), 
the infinite matching $\mathbf M^{\{\emptyset,-2k\}}_{\infty}(\phi)$ starting from $\emptyset$ 
(or any initial condition $Y \in \mathscr Y_2^{\infty}$) 
at time $-2k$, and letting $k$ go to infinity. The second one, by constructing the infinite matching $\mathbf M^{\{\emptyset,-2k+1\}}_{\infty}(\phi)$ starting from $\emptyset$ 
(or any $Y \in \mathscr Y_2^{\infty}$) at time $-2k+1$, and letting $k$ go to infinity. To summarize, 

\begin{corollary}
\label{cor:biinfdouble}
Let $(G,\mu,\phi)$ be a GM model satisfying the assumptions of Theorem \ref{thm:main}. 
Then, on the original time scale there exist exactly two perfect bi-infinite $\phi$-matchings. They coincide a.s. respectively with the 
two stationary perfect bi-infinite $\phi$-matchings induced by the two sequences $\suitez{W^*_n}$ and $\suitez{W^{**}_n}$ defined by (\ref{eq:defWfinal0}) and (\ref{eq:defWfinal1}), 
and can be obtained, for any $Y \in \mathscr Y_2^{\infty}$, by constructing 
the infinite matchings $\mathbf M^{\{Y,-2k\}}_{\infty}(\phi)$ and $\mathbf M^{\{Y,-2k+1\}}_{\infty}(\phi)$ starting from $Y$ respectively at time $-2k$ and $-2k+1$, and letting $k$ go to infinity.  
\end{corollary}

See an example on Figure \ref{Fig:perfectmatch}. 

\begin{figure}[h!]
\begin{center}
\begin{tikzpicture},1
\draw[-] (0,1) -- (11,1);
\draw[-] (0,0) -- (11,0);
\draw[->, thin] (0,0) .. controls +(up:0.5cm)  .. (0.5,0);
\fill (0,1) circle (2pt) node[below] {\small{4}};
\fill (0.5,1) circle (2pt) node[below] {\small{2}};
\draw[->, thin] (-0.5,1) .. controls +(up:0.5cm)  .. (0.5,1);
\fill (0,0) circle (2pt) node[below] {\small{4}};
\fill (0.5,0) circle (2pt) node[below] {\small{2}};
\fill (1,1) circle (2pt) node[below] {\small{3}};
\draw[->, thin] (0,1) .. controls +(up:0.5cm)  .. (1,1);
\fill (1.5,1) circle (2pt) node[below] {\small{2}};
\draw[->, thin] (1.5,1) .. controls +(up:0.5cm)  .. (2,1);
\fill (1,0) circle (2pt) node[below] {\small{3}};
\draw[->, thin] (1,0) .. controls +(up:0.5cm)  .. (1.5,0);
\fill (1.5,0) circle (2pt) node[below] {\small{2}};
\fill (2,1) circle (2pt) node[below] {\small{1}};
\fill (2.5,1) circle (2pt) node[below] {\small{1}};
\draw[->, thin] (2.5,1) .. controls +(up:0.5cm)  .. (3.5,1);
\fill (2,0) circle (2pt) node[below] {\small{1}};,1
\draw[->, thin] (2,0) .. controls +(up:0.5cm)  .. (3.5,0);
\fill (2.5,0) circle (2pt) node[below] {\small{1}};
\fill (3,1) circle (2pt) node[below] {\small{3}};
\draw[->, thin] (3,1) .. controls +(up:0.5cm)  .. (4,1);
\fill (3.5,1) circle (2pt) node[below] {\small{2}};
\fill (3,0) circle (2pt) node[below] {\small{3}};
\draw[->, thin] (3,0) .. controls +(up:0.5cm)  .. (4,0);
\fill (3.5,0) circle (2pt) node[below] {\small{2}};
\fill (4,1) circle (2pt) node[below] {\small{4}};
\fill (4.5,1) circle (2pt) node[below] {\small{3}};
\draw[->, thin] (4.5,1) .. controls +(up:0.5cm)  .. (5,1);
\fill (4,0) circle (2pt) node[below] {\small{4}};
\draw[->, thin] (4.5,0) .. controls +(up:0.5cm)  .. (5,0);
\fill (4.5,0) circle (2pt) node[below] {\small{3}};
\fill (5,1) circle (2pt) node[below] {\small{4}};
\fill (5.5,1) circle (2pt) node[below] {\small{1}};
\draw[->, thin] (5.5,1) .. controls +(up:0.5cm)  .. (7.5,1);
\fill (5,0) circle (2pt) node[below] {\small{4}};
\draw[->, thin] (6.5,0) .. controls +(up:0.5cm)  .. (7,0);
\fill (5.5,0) circle (2pt) node[below] {\small{1}};
\fill (6,1) circle (2pt) node[below] {\small{1}};
\draw[->, thin] (6.5,1) .. controls +(up:0.5cm)  .. (7,1);
\fill (6.5,1) circle (2pt) node[below] {\small{3}};
\fill (6,0) circle (2pt) node[below] {\small{1}};
\draw[->, thin] (2.5,0) .. controls +(up:0.5cm)  .. (7.5,0);
\fill (6.5,0) circle (2pt) node[below] {\small{3}};
\fill (7,1) circle (2pt) node[below] {\small{4}};
\draw[->, thin] (6,1) .. controls +(up:0.5cm)  .. (8,1);
\fill (7.5,1) circle (2pt) node[below] {\small{2}};
\fill (7,0) circle (2pt) node[below] {\small{4}};
\draw[->, thin] (5.5,0) .. controls +(up:0.5cm)  .. (8,0);
\fill (7.5,0) circle (2pt) node[below] {\small{2}};
\fill (8,1) circle (2pt) node[below] {\small{2}};
\fill (8.5,1) circle (2pt) node[below] {\small{1}};
\draw[->, thin] (8.5,1) .. controls +(up:0.5cm)  .. (9.5,1);
\fill (8,0) circle (2pt) node[below] {\small{2}};
\draw[->, thin] (8.5,0) .. controls +(up:0.5cm)  .. (9.5,0);
\fill (8.5,0) circle (2pt) node[below] {\small{1}};
\fill (9,1) circle (2pt) node[below] {\small{4}};
\draw[->, thin] (9,1) .. controls +(up:0.5cm)  .. (10,1);
\fill (9.5,1) circle (2pt) node[below] {\small{2}};
\fill (9,0) circle (2pt) node[below] {\small{4}};
\draw[->, thin] (9,0) .. controls +(up:0.5cm)  .. (10,0);
\fill (9.5,0) circle (2pt) node[below] {\small{2}};
\fill (10,1) circle (2pt) node[below] {\small{3}};
\fill (10.5,1) circle (2pt) node[below] {\small{2}};
\draw[->, thin] (10.5,1) .. controls +(up:0.5cm)  .. (11,1);
\fill (10,0) circle (2pt) node[below] {\small{3}};
\draw[->, thin] (10.5,0) .. controls +(up:0.5cm)  .. (11,0);
\fill (10.5,0) circle (2pt) node[below] {\small{2}};
\fill (11,1) circle (2pt) node[below] {\small{1}};
\fill (11,0) circle (2pt) node[below] {\small{1}};
\draw[->] (-0.5,-1) -- (11.5,-1);
\fill (5.5,-1) node[]{$|$} node[below]{$0$};
\end{tikzpicture}
\caption[smallcaption]{Two perfect bi-infinite {\sc fcfm}- (or {\sc ml}-) matchings corresponding to the graph of 
Figure \ref{Fig:paw} and the same input.}
\label{Fig:perfectmatch}
\end{center}
\end{figure}
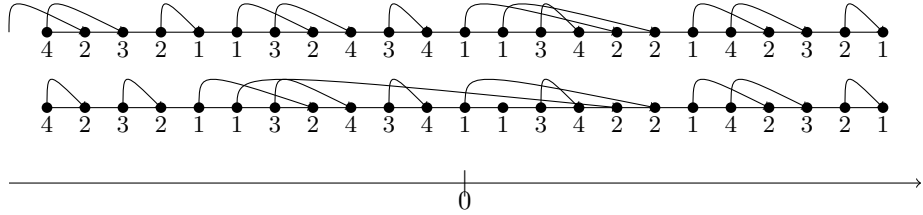

\subsection{Perfect {\sc fcfm}-matchings in reverse time}
\label{FCFMreverse}
If the matching policy is {\sc fcfm}, perfect bi-infinite matchings have an interesting property. 
First, observe that we can complete the `exchange" mechanism introduced in Section 3 of \cite{MBM21}, using construction points as follows: 
Start from a construction point, and then replace all items from left to the right by the copy of the class of their match, on the fly, as soon as they are matched (see again the construction in Section 3 of \cite{MBM21}). 
We illustrate this procedure in Figure \ref{Fig:exchange}, by the completion of the exchanges over two perfectly matched blocks, for the compatibility graph of Figure \ref{Fig:paw}, on a given arrival scenario. 
\begin{figure}[h!]
\begin{center}
\begin{tikzpicture}[scale=0.8]
%
\draw[-] (-1,1) -- (11,1);
\fill (0,1) circle (2pt) node[below] {\small{1}};
\draw[-, thin] (0,1) .. controls +(up:0.5cm)  .. (3,1);
\fill (1,1) circle (2pt) node[below] {\small{3}};
\draw[-, thin] (1,1) .. controls +(up:0.5cm)  .. (2,1);
\fill (2,1) circle (2pt) node[below] {\small{4}};
\fill (3,1) circle (2pt) node[below] {\small{$2$}};
\fill (4,1) circle (2pt) node[below] {\small{3}};
\draw[-, thin] (4,1) .. controls +(up:0.5cm)  .. (7,1);
\fill (5,1) circle (2pt) node[below] {\small{1}};
\draw[-, thin] (5,1) .. controls +(up:0.5cm)  .. (8,1);
\fill (6,1) circle (2pt) node[below] {\small{3}};
\draw[-, thin] (6,1) .. controls +(up:0.5cm)  .. (10,1);
\fill (7,1) circle (2pt) node[below] {\small{2}};
\fill (8,1) circle (2pt) node[below] {\small{2}};
\fill (9,1) circle (2pt) node[below] {\small{1}};
\draw[-, thin] (9,1) .. controls +(up:0.5cm)  .. (11,1);
\fill (10,1) circle (2pt) node[below] {\small{4}};
\fill (11,1) circle (2pt) node[below] {\small{2}};
\draw[-] (-1,-0.5) -- (11,-0.5);
\fill (0,-0.5) circle (2pt) node[below] {\small{$\bar 2$}};
\draw[-, thin] (0,-0.5) .. controls +(up:0.5cm)  .. (3,-0.5);
\fill (1,-0.5) circle (2pt) node[below] {\small{$\bar 4$}};
\draw[-, thin] (1,-0.5) .. controls +(up:0.5cm)  .. (2,-0.5);
\fill (2,-0.5) circle (2pt) node[below] {\small{$\bar 3$}};
\fill (3,-0.5) circle (2pt) node[below] {\small{$\bar 1$}};
\fill (4,-0.5) circle (2pt) node[below] {\small{$\bar 2$}};
\draw[-, thin] (4,-0.5) .. controls +(up:0.5cm)  .. (7,-0.5);
\fill (5,-0.5) circle (2pt) node[below] {\small{$\bar 2$}};
\draw[-, thin] (5,-0.5) .. controls +(up:0.5cm)  .. (8,-0.5);
\fill (6,-0.5) circle (2pt) node[below] {\small{$\bar 4$}};
\draw[-, thin] (6,-0.5) .. controls +(up:0.5cm)  .. (10,-0.5);
\fill (7,-0.5) circle (2pt) node[below] {\small{$\bar 3$}};
\fill (8,-0.5) circle (2pt) node[below] {\small{$\bar 1$}};
\fill (9,-0.5) circle (2pt) node[below] {\small{$\bar 2$}};
\draw[-, thin] (9,-0.5) .. controls +(up:0.5cm)  .. (11,-0.5);
\fill (10,-0.5) circle (2pt) node[below] {\small{$\bar 3$}};
\fill (11,-0.5) circle (2pt) node[below] {\small{$\bar 1$}};
\end{tikzpicture}
\caption[smallcaption]{Top: Two blocks matched in 
{\sc fcfm}. Bottom: completion of the exchanges by matchings.} 
\label{Fig:exchange}
\end{center}
\end{figure}
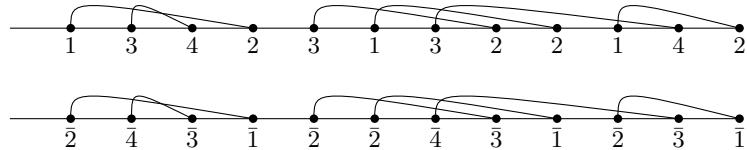

Now observe the following: between any pair of construction points (if any), after completion of the exchanges on the perfectly matched block by {\sc fcfm}, 
by reading the arrivals on the matched block from right to left, we see nothing but a {\sc fcfm} matching of the items of classes in 
$\td\maV$ (the set of copies of classes in $\maV$). To prove this, let the four nodes $i$, $j$, $k$ and $\ell$ be such that in $G$, $i \v k$, $j \v k$ and $i \v \ell$, 
and suppose that, after the exchange, four copies $\td i$, $\td j$, $\td k$ and $\td{\ell}$ are read in that order, in reverse time, i.e. from right to left. 
Let us also assume that the {\sc fcfm} rule in reverse time is violated on this quadruple: then 
the $\td k$ item is matched with the $\td j$ item while the $\td i$ item is still unmatched, and then the latter item is matched with the $\td{\ell}$ item. This occurs if and only if, in direct time, the four items of 
classes $i$, $j$, $k$ and $\ell$ arrive in that order, and the $k$ item choses the $j$ item over the $i$ item for its match, and then the unmatched $i$ item is matched with the $\ell$ item. 
This violates in turn the {\sc fcfm} policy, according to which the $k$ item should have been matched with the $i$ item instead of the $j$ item. Hence the assertion above: over any perfectly matched block 
in {\sc fcfm}, the block of exchanged items read in reverse time is also perfectly matched in {\sc fcfm} - see the bottom display of Figure \ref{Fig:exchange}. 

Now assume that $G$ is non-bipartite and that $\mu\in \textsc{Ncond}(G)$. Hence, there exists a.s. infinitely many construction points on $\mathbb Z$, and 
exactly two perfect bi-infinite {\sc fcfm}-matchings of the entered items. Generalizing the above observation to all perfectly matched blocks on $\mathbb Z$ between pairs of construction points, as {\sc fcfm} is sub-additive we conclude from Corollary \ref{cor:biinfdouble} that there exist exactly two perfect 
bi-infinite {\sc fcfm}-matchings of the exchanged items in reverse time, corresponding respectively to the two aforementioned perfect {\sc fcfm}-matchings in direct time, after complete exchanges over blocks, read from right to left. 

\section[Perfect simulation]{A perfect simulation algorithm for finite-capacity systems}
\label{sec:PW}
The celebrated  ``Coupling From The Past" (CFTP for short) algorithm introduced by Propp and Wilson \cite{PW96} allows to sample values of the considered Markov chain (here, the buffer content process $\suite{W_n}$) from its stationary distribution, even if the latter is not know explicitly. 
As formalized in Section 4 of \cite{FT98}, CFTP, and in particular, the coalescence of the Markov chains corresponding to a given set of initial conditions, can be phrased in terms of strong backwards coupling convergence of stochastic recursions, as in the present framework. We show hereafter how the sub-additivity of matching models can allow to construct a simple CFTP algorithm, for models 
having a finite buffer. 

\subsection{A finite-buffer system}
\label{subsec:modelfinite}
Consider a finite-buffer version of the (even) system at hand: namely, we suppose that there is a finite buffer 
of size $2C$, where $C\in\N_+$. Then, the matching model has exactly the dynamics that is specified in Section \ref{sec:model}, aside from the following situation: whenever the buffer reaches its capacity $2C$ and a group of two successive incoming items are incompatible with any element in the buffer and with one another, they are both discarded from the system. In detail, suppose that the buffer is full, and that 
two entering item of respective classes, say, $i$ and then $j$, are proposed for matching to the elements in the buffer. Then we face the following alternative: 
\begin{enumerate}
\item If the $i$-item finds a match in the buffer, then the corresponding match is executed right away and creates a new free spot in the buffer. Then: 
	\begin{enumerate}
	\item[(1a)] either the $j$-item also finds a match in the buffer, and the later is executed right away. 
	\item[(1b)] else, the $j$-item is stored in the buffer, occupying the new spot.  
	\end{enumerate}
\item Else, the $i$-item is not matched, but is blocked. Then, we investigate whether $j$ is matchable in the buffer or with the $i$-item, following the prescribed matching policy, and: 
	\begin{enumerate}
	\item[(2a)] If the $j$-item is matched with the $i$-item, 
	the two incoming items are matched together and leave the system;
	\item[(2b)] If the $j$-item is matched with another item in the buffer, this creates a free spot, and the $i$-item 
	is stored in it. 
	\item[(2c)] If the $j$-item is not matchable at all, then the two incoming items cannot be stored in the (full) buffer, 
	and are both discarded right away. 
	\end{enumerate}
\end{enumerate}
We append a superscript $C$ to the processes relative to this modified model. In particular, we let 
$U^C_n$ be the even buffer content at time $n$, and let $W^{C}_\phi(v)$ be the buffer content of the system fed by the input $v$.

\subsection{Perfect simulation at even times}

It is clear that $\suite{U^C_n}$ is a Markov chain of finite state space $\mathbb W_2(C)$. 
As it is clearly irreducible, $\suite{U^C_n}$ is ergodic on the finite state space $\mathbb W_2(C)$. However, determining explicitly the stationary distribution of that chain is a priori out of reach for a general matching policy $\phi$. 
On them aims at the construction of a perfect sampling algorithm by a CFTP algorithm, whose general procedure is as follows: simulating versions of the Markov chain $\suite{U^C_n}$ starting from {all} possible initial states in $\mathbb W_2(C)$, and fed by the same input, by starting at a time point that is far away in the past, to guarantee that all versions of the Markov chain coalesce (and then, coincide forever) before time 0. Then, see e.g. \cite{PW96,FT98}, the value of any version of the chains at time zero provides a sample from the stationary distribution of the even buffer content of the model. Then, re-iterating this procedure for a large number of samples then allows the estimation of mean characteristics of the system in steady state, e.g. using Monte Carlo methods. 

It is also well known that such CFTP algorithm starting from {\em all} possible states becomes prohibitively hard in terms of complexity, as the size of the state space gets large. Instead, one is typically led to propose other, simpler, perfect simulation algorithms. 
The typical cases in the literature where such simplified algorithms exist are the monotone CFTP of Propp and Wilson \cite{PW96}, the Dominated Coupling From the Past of Kendall (see e.g. \cite{Ken98,KM00}), or algorithms based on bounding chains \cite{huber2004}. 
But it is immediate to observe that, in the present model, the chain $\suite{U_n}$ is {\em not} monotonic, and in fact, it is {\em a priori} hard to construct an ordering on the sate space $\mathbb W_2(p)$ that would allow the domination or the bounding by a simpler chain. Instead, we use an approach that is closely related to the control method of \cite{MM22}, and to the small set approach of  \cite{MG98,Wil00}: Checking that the arrival scenario includes strong erasing words, guaranteeing that all versions of the Markov chain coalesce to the empty state. Then, the unique version of the chain can be simulated from that instant on, instead of simulating all trajectories of the chains for all possible initial values. 


Specifically, suppose that $(G,\phi)$ admits strong erasing words, and suppose that we dispose of a dictionary  
$\mathscr S_{2}(p)$ of $2C$-strong erasing words of length $2p$. Let us consider Algorithm \ref{algo1}. 
\begin{algorithm}
\caption{Simulation of the stationary probability of $U$ - Matching model with finite buffer.}
\KwData {A probability distribution $\nu$ on $\mathbb{V}$, a set $\mathscr S_2(p)$ of strong erasing words of length $2p$} 
$z\gets \emptyset$ \;
\For{$j \gets -p \ \KwTo \ 1 $}
{{ \bfseries{draw} $(v_{2j},\zeta_{2j},v_{2j+1},\zeta_{2j+1})$  \bfseries{from} $\mu \otimes \nu_\phi\otimes\mu\otimes\nu_\phi$ \;}
$z \gets zv_j$ \tcc*[l]{We construct the arrival scenario from time $-2p$ to time $-1$ and set $z=z_1\cdots z_{2p}=v_{-2p}\cdots v_{-1}$ }}
$i \gets -p$ \;
\While{$z \,\mbox{{\em is not an element of $\mathscr S_2(p)$}}$}
{ $i \gets i-1$ \;
\bfseries{draw} $(v_{2i},\zeta_{2i},v_{2i+1},\zeta_{2j+1})$  \bfseries{from} $\mu \otimes \nu_\phi\otimes\mu\otimes\nu_\phi$  \;
$z \gets v_{2i}v_{2i+1}z_2z_3\cdots z_{2p-2}$ \tcc*[l]{We update the last $2p$ arrivals backwards in time, $2$ by $2$, until we reach the first strong erasing word of length $2p$}}
$i \gets i+p$ \;
$U \gets \emptyset $ \tcc*[l]{We reset the system just after the arrival of the strong erasing word, and assign to $U$ the empty set.}
\While{$i < 0$}
{$U \longleftarrow  \left(U \odot_\phi (v_{2i},\zeta_{2i})\right) \odot_\phi (v_{2i+1},\zeta_{2i+1}) $ \tcc{ We now transition the chain $U$ at even times until time $0$.}
$i \longleftarrow i+1$ \;}
\KwRet{$U$}.\\
\text{ }\\
\label{algo1}
\end{algorithm}

Then we have the following result, 

\begin{theorem}
\label{thm:mainC}
Suppose that for some $p\in\N_+$, there exists a $2C$-strong erasing word of $\mathbb W_2(p)$. 
Then, the Markov chain $\suite{U_n}$ is uniformly ergodic. Moreover, Algorithm \ref{algo1} terminates almost surely, and its output is sampled from the stationary distribution $\pi$ on 
$\mathbb W_2(C)$. 
\end{theorem}

\begin{proof}
Let $z$ be a $2C$-strong erasing word of length $2p$ for $(G,\phi)$. For all $w\in\mathbb W_2(C)$ and for all compatible 
$\zeta,\zeta'\in\mathcal S^*$, 
from (ii) of Definition \ref{def:strongerase} we get that 
\[\left|W_\phi(wz_1\cdots z_{2r},\zeta\zeta'_{1}\cdots\zeta'_{2r})\right| \le \left|W_\phi(w,\zeta)\right|\le 2C,\] 
for any even prefix $z_1\cdots z_{2r}$ of $z$. In particular, the $2C$-finite buffer system operates exactly like the corresponding infinite buffer system 
along the input $(z,\zeta')$, because no element is lost because of a full buffer. In particular, conditional on the input being $(z,\zeta')$ over the time interval $\llbracket 1,2p \rrbracket$, with obvious notation 
we get that 
\begin{equation}
\label{eq:blade}
U^{\{w\},C}_{p}=U^{\{w\}}_{p}=W_\phi(wz,\zeta\zeta')=\emptyset,
\end{equation}
applying (i) of Definition \ref{def:strongerase} in the last equality. Let us now set 
\[N^-:=\inf\left\{n\ge 0:\mbox{for some }p,\,V_{-2n}V_{-2n+1}\cdots V_{-2(n-p)-1} \mbox{ is a $2C$-strong erasing word for $(G,\phi)$}\right\},\]
which is set to $\infty$ if the above set is empty. 
Then any model started before time $-2N^-$ is empty at time $-2(N^--p)$. This implies that 
for all $n\ge N^-$, for all $w,w'\in \mathbb W_2(C)$, we get that 
\[U^{\{w\},C}_n\circ\theta^{-n} = U^{\{w'\},C}_n\circ\theta^{-n}=U^{\{\emptyset\},C}_n\circ\theta^{-(N^--p)},\]
that is, the buffer content at time 0 of a system started empty at time $-(N^--p)$ (on the even time scale). 
In other words, $N^-$ is a backwards coalescence time for the $2C$-finite buffer system. Clearly, from the IID assumptions, it is a.s. finite whenever 
$(G,\phi)$ admits a $2C$-strong erasing words. In that case, the backwards coalescence time is successful (in the terminology of \cite{FT98}), and so 
the output of Algorithm \ref{algo1}, which is precisely $U^{\{\emptyset\},C}_n\circ\theta^{-(N^--p)}$, is sampled from the unique stationary 
distribution of $\suite{U_n}$ from Theorem 4.1 in \cite{FT98}. From Theorem 4.2. in \cite{FT98}, this is equivalent to saying that the chain is uniformly ergodic. 
\end{proof}

Algorithm \ref{algo1} terminates as soon as $2p$ consecutive arrivals form a $2C$-strong erasing word of our dictionary 
$\mathscr S_2(p)$. Then (and only then), the buffer content at even times $\suite{U_n}$ is simulated until time zero, and its value at that time is precisely distributed from the stationary distribution of $\suite{U_n}$. Clearly, the efficiency of Algorithm \ref{algo1} resides in the 
ability to determine the largest possible dictionary $\mathscr S_2(p)$ of $2C-$strong erasing words of length $2p$. For this, one is led to apply combinatorial arguments similar to the proof of Proposition \ref{pro:erasing2}. 
Clearly, the very existence of such a dictionary $\mathscr S_2(p)$ and its construction are highly dependent on the graph geometry, and a more complete study in that direction is left for future research. 

Notice, however, that the argument in the proof of Proposition \ref{pro:erasing2} suggests a crucial shortcut to the `brute force' numerical task of testing whether {all} arrival scenarios of length $2p$ in a given set of interest are $2C$-erasing words. If $\phi$ is sub-additive, it is indeed sufficient to fix a given length $q$, and checking the $2$-strong erasing property for all words of length $2q$ in a given subset of words of interest. Then, 
all combinations $z^1\cdots z^C$ of $C$ (possible equal) such $2$-strong erasing words is itself, a $2C$-strong erasing word of length $2Cq$ for $(G,\phi)$. Then, we can apply Algorithm \ref{algo1} to $p=Cq$, and taking as dictionary $\mathscr S_{2}(Cq)$, the set of all such combinations of $2$-strong erasing words.

\providecommand{\bysame}{\leavevmode\hbox to3em{\hrulefill}\thinspace}
\providecommand{\MR}{\relax\ifhmode\unskip\space\fi MR }
\providecommand{\MRhref}[2]{%
  \href{http://www.ams.org/mathscinet-getitem?mr=#1}{#2}
}
\providecommand{\href}[2]{#2}

\newpage
\appendix

\section{Proof of Proposition \ref{pro:erasing}}
\label{sec:appendixA}
As will be clear below, the arguments of this proof do not depend on the drawn lists of preferences, as long as they are fixed upon arrival. For notational convenience, we thus skip this parameter from all notations 
(i.e. we write for instance $W_\phi(u)$ instead of $W_\phi(u,\varsigma)$, and so on). We first show that any admissible word of size 2 admits an erasing word $y$; so let us consider a word $ij$ where $i \pv j$. 
%


\medskip

As $G$ is connected, $i$ and $j$ are connected at distance, say, $p\ge 2$, i.e. there exists a shortest path $i \v i_1 \v ... \v i_{p-1} \v j$ connecting $i$ to $j$. 
If $p$ is odd, then just set $y=i_1i_2...i_{p-1}$. Clearly, $W_\phi(w)=\emptyset$ and as the path has no short-cut, in $M_\phi(ijz)$ $i_1$ is matched with $i$, $i_3$ is matched with $i_2$, and so on, 
until $i_{p-1}$ is matched with $j$. So $W_\phi(ijy)=\emptyset$, and (\ref{eq:deferase}) follows. 

We now assume that $p$ is even. Set $y^1=i_1i_2...i_{p-1}i_{p-1}$. Then,  
in $M_\phi(ijy^1)$ $i_1$ is matched with $i$, $i_3$ with $i_2$, and so on, until both $j$ and $i_{p-2}$ are matched with an $i_{p-1}$ item. So $W_\phi(ijy^1)=\emptyset$, however 
$W_\phi(y^1) = i_{p-1}i_{p-1}$. But as $G$ is non-bipartite, it contains an odd cycle. Thus (see e.g. the proof of Lemma 3 in \cite{MoyPer17}) there necessarily exists an {\em induced} odd cycle in $G$ (meaning that no shortest path exist between two of its nodes), 
say of length $2r+1$, $r \ge 1$. As $G$ is connected, there exists a path connecting $i_{p-1}$ to any element of the latter cycle. Take the shortest one (which may intersect with the path between $i$ to $j$, or coincide with a part of it), and denote it $i_{p-1} \v j_1 \v j_2 \v ... \v j_q \v k_1$, where $k_1$ is the first element of the latter path belonging to the odd cycle, and by $k_1 \v k_2 \v ... \v k_{2r+1} \v k_1$, the elements of the cycle. 
See an example in Figure \ref{Fig:erasingpath}. 

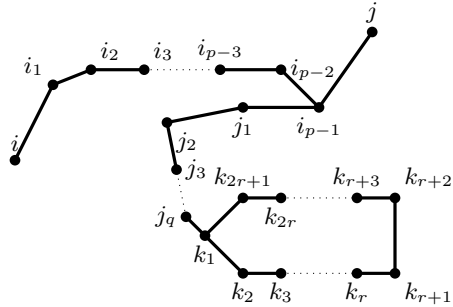
\begin{figure}[h!]
\begin{center}
\begin{tikzpicture}
\draw[-,very thick] (0.5,0) -- (1,1);
\draw[-,very thick] (1,1) -- (1.5,1.2);
\draw[-,very thick] (1.5,1.2) -- (2.2,1.2);
\draw[-,dotted] (2.2,1.2) -- (3.2,1.2);
\draw[-,very thick] (3.2,1.2) -- (4,1.2);
\draw[-,very thick] (4,1.2) -- (4.5,0.7);
\draw[-,very thick] (4.5,0.7) -- (5.2,1.7);
\draw[-,very thick] (4.5,0.7) -- (3.5,0.7);
\draw[-,very thick] (3.5,0.7) -- (2.5,0.5);
\draw[-,very thick] (2.5,0.5) -- (2.625,-0.125);
\draw[-,dotted] (2.5,0.5) -- (2.75,-0.75);
\draw[-,very thick] (2.75,-0.75) -- (3,-1);
\draw[-,very thick] (3,-1) -- (3.5,-0.5);
\draw[-,very thick] (3,-1) -- (3.5,-1.5);
\draw[-,very thick] (3.5,-1.5) -- (4,-1.5);
\draw[-,dotted] (4,-1.5) -- (5.5,-1.5);
\draw[-,dotted] (5,-1.5) -- (5.5,-1.5);
\draw[-,very thick] (5.5,-1.5) -- (5.5,-0.5);
\draw[-,very thick] (3.5,-0.5) -- (4,-0.5);
\draw[-,dotted] (4,-0.5) -- (5.5,-0.5);
\draw[-,dotted] (5,-1.5) -- (5.5,-1.5);
\draw[-,very thick] (5,-0.5) -- (5.5,-0.5);
\draw[-,very thick] (5,-1.5) -- (5.5,-1.5);
\fill  (0.5,0) circle (2pt) node[above] {\small{$i$}} ;
\fill (1,1) circle (2pt) node[above left] {\small{$i_1$}} ;
\fill (1.5,1.2) circle (2pt) node[above right] {\small{$i_2$}} ;
\fill (2.2,1.2) circle (2pt) node[above right] {\small{$i_3$}} ;
\fill (3.2,1.2) circle (2pt) node[above] {\small{$i_{p-3}$}} ;
\fill (4,1.2) circle (2pt) node[right] {\small{$i_{p-2}$}} ;
\fill (4.5,0.7) circle (2pt) node[below] {\small{$i_{p-1}$}} ;
\fill (5.2,1.7) circle (2pt) node[above] {\small{$j$}} ;
\fill (3.5,0.7) circle (2pt) node[below] {\small{$j_{1}$}} ;
\fill (2.5,0.5) circle (2pt) node[below right] {\small{$j_{2}$}} ;
\fill (2.625,-0.125) circle (2pt) node[right] {\small{$j_{3}$}} ;
\fill (2.75,-0.75) circle (2pt) node[left] {\small{$j_{q}$}} ;
\fill (3,-1) circle (2pt) node[below] {\small{$k_{1}$}} ;
\fill (3.5,-0.5) circle (2pt) node[above] {\small{$k_{2r+1}$}} ;
\fill (3.5,-1.5) circle (2pt) node[below] {\small{$k_{2}$}} ;
\fill (4,-0.5) circle (2pt) node[below] {\small{$k_{2r}$}} ;
\fill (5,-0.5) circle (2pt) node[above] {\small{$k_{r+3}$}} ;
\fill (4,-1.5) circle (2pt) node[below] {\small{$k_{3}$}} ;
\fill (5.5,-0.5) circle (2pt) node[above right] {\small{$k_{r+2}$}} ;
\fill (5.5,-1.5) circle (2pt) node[below right] {\small{$k_{r+1}$}} ;
\fill (5,-1.5) circle (2pt) node[below] {\small{$k_{r}$}} ;
\end{tikzpicture}
\caption[smallcaption]{The path from $i$ to $j$ and then to an odd cycle}
\label{Fig:erasingpath}
\end{center}
\end{figure}

Then set 
\[y^2 = j_1 j_1 j_2 j_2 ... j_q j_q k_1 k_1 k_2 k_3 ... k_{2r} k_{2r+1}.\]
We are in the following alternative: 
\begin{itemize}
\item if $q$ is even, then in $M_\phi(y^1y^2)$ the two nodes $i_{p-1}$ are matched with the two nodes $j_1$, the two $j_2$ with the two $j_3$, and so on, until the two $j_{q}$ are matched with the two $k_1$, and then, as the 
cycle is induced, $k_2$ is matched with $k_3$, $k_4$ with $k_5$ and so on, until $k_{2p}$ is matched with $k_{2p+1}$. 
On the other hand, in $M_\phi(y^2)$, the two $j_1$ are matched with the two $j_2$, the two $j_3$ with the two $j_4$, and so on, until the two $j_{q-1}$ are matched with the two $j_q$. 
Then, a $k_1$ is matched with $k_2$, $k_3$ with $k_4$ and so on, until $k_{2p-1}$ is matched with $k_{2p}$ and $k_{2p+1}$ is matched with the remaining $k_1$. 
\item if $q$ is odd, then the edges of $M_\phi(y^1y^2)$ are as in the first case, until the two nodes $j_{q-1}$ are matched with the two nodes $j_q$. But then, whatever $\phi$ is, 
one of the two nodes $k_1$ is matched with $k_2$, $k_3$ with $k_4$, and so on, until $k_{2p-1}$ is matched with $k_{2p}$, and $k_{2p+1}$ is matched with the remaining $k_1$. 
Also, in $M_\phi(y^2)$, the two $j_1$ are matched with the two $j_2$, the two $j_3$ with the two $j_4$, and so on, until the two $j_{q-2}$ are matched with the two $j_{q-1}$.  
Then, the two $j_q$ are matched with the two $k_1$, $k_2$ is matched with $k_3$, and so on, until $k_{2p}$ is matched with $k_{2p+1}$.
\end{itemize} 
In both cases, we obtain that both $W_\phi(y^1y^2)=\emptyset$ and that $W_\phi(y^2)=\emptyset$. In particular, as $W_\phi(ijy^1)=\emptyset$ we have $W_\phi(ijy^1y^2)=\emptyset$. 
Therefore $y=y^1y^2$ is an erasing word for $ij$. See an example in Figure \ref{Fig:erasingword}. 

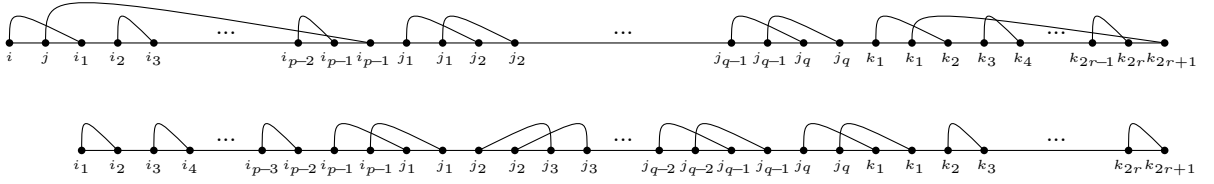
\begin{figure}[h!]
\begin{center}
\begin{tikzpicture}[scale=0.95]
\draw[-] (-0.5,0) -- (15.5,0);
\fill (-0.5,0) circle (1.5pt) node[below] {\tiny{$i$}};
\fill (0,0) circle (1.5pt) node[below] {\tiny{$j$}};
\fill (0.5,0) circle (1.5pt) node[below] {\tiny{$i_1$}};
\draw[-, thin] (-0.5,0) .. controls +(up:0.5cm)  .. (0.5,0);
\fill (1,0) circle (1.5pt) node[below] {\tiny{$i_2$}};
\fill (1.5,0) circle (1.5pt) node[below] {\tiny{$i_3$}};
\draw[-, thin] (1,0) .. controls +(up:0.5cm)  .. (1.5,0);
\fill (2.5,0) node[above] {\small{...}};
\fill (3.5,0) circle (1.5pt) node[below] {\tiny{$i_{p\!-\!2}$}};
\fill (4,0) circle (1.5pt) node[below] {\tiny{\,\,$i_{p\!-\!1}$}};
\draw[-, thin] (3.5,0) .. controls +(up:0.5cm)  .. (4,0);
\fill (4.5,0) circle (1.5pt) node[below] {\tiny{\,\,$i_{p\!-\!1}$}};
\draw[-, thin] (0,0) .. controls +(up:0.75cm)  .. (4.5,0);
\fill (5,0) circle (1.5pt) node[below] {\tiny{$j_1$}};
\fill (5.5,0) circle (1.5pt) node[below] {\tiny{\,\,$j_1$}};
\fill (6,0) circle (1.5pt) node[below] {\tiny{$j_2$}};
\fill (6.5,0) circle (1.5pt) node[below] {\tiny{\,\,$j_2$}};
\draw[-, thin] (5,0) .. controls +(up:0.5cm)  .. (6,0);
\draw[-, thin] (5.5,0) .. controls +(up:0.5cm)  .. (6.5,0);
\fill (8,0) node[above] {\small{...}};
\fill (9.5,0) circle (1.5pt) node[below] {\tiny{$j_{q\!-\!1}$}};
\fill (10,0) circle (1.5pt) node[below] {\tiny{\,\,$j_{q\!-\!1}$}};
\fill (10.5,0) circle (1.5pt) node[below] {\tiny{$j_q$}};
\fill (11,0) circle (1.5pt) node[below] {\tiny{\,\,$j_q$}};
\draw[-, thin] (9.5,0) .. controls +(up:0.5cm)  .. (10.5,0);
\draw[-, thin] (10,0) .. controls +(up:0.5cm)  .. (11,0);
\fill (11.5,0) circle (1.5pt) node[below] {\tiny{$k_1$}};
\fill (12,0) circle (1.5pt) node[below] {\tiny{\,\,$k_1$}};
\fill (12.5,0) circle (1.5pt) node[below] {\tiny{\,\,$k_2$}};
\draw[-, thin] (11.5,0) .. controls +(up:0.5cm)  .. (12.5,0);
\fill (13,0) circle (1.5pt) node[below] {\tiny{\,\,$k_3$}};
\fill (13.5,0) circle (1.5pt) node[below] {\tiny{\,\,$k_4$}};
\draw[-, thin] (13,0) .. controls +(up:0.5cm)  .. (13.5,0);
\fill (14,0) node[above] {\small{...}};
\fill (14.5,0) circle (1.5pt) node[below] {\tiny{$k_{2r\!-\!1}$}};
\fill (15,0) circle (1.5pt) node[below] {\tiny{\,\,$k_{2r}$}};
\draw[-, thin] (14.5,0) .. controls +(up:0.5cm)  .. (15,0);
\fill (15.5,0) circle (1.5pt) node[below] {\tiny{\,\,\,\,\,$k_{2r+1}$}};
\draw[-, thin] (12,0) .. controls +(up:0.5cm)  .. (15.5,0);
\draw[-] (0.5,-1.5) -- (15.5,-1.5);
\fill (0.5,-1.5) circle (1.5pt) node[below] {\tiny{$i_1$}};
\fill (1,-1.5) circle (1.5pt) node[below] {\tiny{$i_2$}};
\draw[-, thin] (0.5,-1.5) .. controls +(up:0.5cm)  .. (1,-1.5);
\fill (1.5,-1.5) circle (1.5pt) node[below] {\tiny{$i_3$}};
\fill (2,-1.5) circle (1.5pt) node[below] {\tiny{$i_4$}};
\draw[-, thin] (1.5,-1.5) .. controls +(up:0.5cm)  .. (2,-1.5);
\fill (2.5,-1.5) node[above] {\small{...}};
\fill (3,-1.5) circle (1.5pt) node[below]{\tiny{$i_{p\!-\!3}$}};
\fill (3.5,-1.5) circle (1.5pt) node[below] {\,\tiny{$i_{p\!-\!2}$}};
\draw[-, thin] (3,-1.5) .. controls +(up:0.5cm)  .. (3.5,-1.5);
\fill (4,-1.5) circle (1.5pt) node[below] {\tiny{\,\,$i_{p\!-\!1}$}};
\fill (4.5,-1.5) circle (1.5pt) node[below] {\tiny{\,\,\,\,$i_{p\!-\!1}$}};
\fill (5,-1.5) circle (1.5pt) node[below] {\tiny{\,$j_1$}};
\fill (5.5,-1.5) circle (1.5pt) node[below] {\tiny{\,\,$j_1$}};
\draw[-, thin] (4,-1.5) .. controls +(up:0.5cm)  .. (5,-1.5);
\draw[-, thin] (4.5,-1.5) .. controls +(up:0.5cm)  .. (5.5,-1.5);
\fill (6,-1.5) circle (1.5pt) node[below] {\tiny{$j_2$}};
\fill (6.5,-1.5) circle (1.5pt) node[below] {\tiny{\,\,$j_2$}};
\fill (7,-1.5) circle (1.5pt) node[below] {\tiny{$j_3$}};
\fill (7.5,-1.5) circle (1.5pt) node[below] {\tiny{\,\,$j_3$}};
\draw[-, thin] (7,-1.5) .. controls +(up:0.5cm)  .. (6,-1.5);
\draw[-, thin] (7.5,-1.5) .. controls +(up:0.5cm)  .. (6.5,-1.5);
\fill (8,-1.5) node[above] {\small{...}};
\fill (8.5,-1.5) circle (1.5pt) node[below]{\tiny{$j_{q\!-\!2}$}};
\fill (9,-1.5) circle (1.5pt) node[below] {\,\tiny{$j_{q\!-\!2}$}};
\fill (9.5,-1.5) circle (1.5pt) node[below] {\tiny{\,\,$j_{q\!-\!1}$}};
\fill (10,-1.5) circle (1.5pt) node[below] {\tiny{\,\,\,\,$j_{q\!-\!1}$}};
\draw[-, thin] (8.5,-1.5) .. controls +(up:0.5cm)  .. (9.5,-1.5);
\draw[-, thin] (9,-1.5) .. controls +(up:0.5cm)  .. (10,-1.5);
\fill (10.5,-1.5) circle (1.5pt) node[below] {\tiny{$j_q$}};
\fill (11,-1.5) circle (1.5pt) node[below] {\tiny{\,\,$j_q$}};
\fill (11.5,-1.5) circle (1.5pt) node[below] {\tiny{$k_1$}};
\fill (12,-1.5) circle (1.5pt) node[below] {\tiny{\,\,$k_1$}};
\draw[-, thin] (10.5,-1.5) .. controls +(up:0.5cm)  .. (11.5,-1.5);
\draw[-, thin] (11,-1.5) .. controls +(up:0.5cm)  .. (12,-1.5);
\fill (12.5,-1.5) circle (1.5pt) node[below] {\tiny{\,\,$k_2$}};
\fill (13,-1.5) circle (1.5pt) node[below] {\tiny{\,\,$k_3$}};
\draw[-, thin] (12.5,-1.5) .. controls +(up:0.5cm)  .. (13,-1.5);
\fill (14,-1.5) node[above] {\small{...}};
\fill (15,-1.5) circle (1.5pt) node[below] {\tiny{$k_{2r}$}};
\fill (15.5,-1.5) circle (1.5pt) node[below] {\tiny{\,\,\,\,$k_{2r+1}$}};
\draw[-, thin] (15,-1.5) .. controls +(up:0.5cm)  .. (15.5,-1.5);
\end{tikzpicture}
\caption[smallcaption]{The two perfect matchings $M_{\textsc{fcfm}}(ijy^1y^2)$ and $M_{\textsc{fcfm}}(y^1y^2)$, for an even $p$ and an odd $q$.}
\label{Fig:erasingword}
\end{center}
\end{figure}

\medskip
We now consider any word $u \in \mathbb W_2$, say $u=u_1u_2...u_{2r_1}$. First, as we just proved, there exists an erasing word, say $z^{1}$, for the two-letter word $u_{2r_1-1}u_{2r_1}$. 
In particular, we have that $W_\phi\left(u_{2r_1-1}u_{2r_1}z^1\right)=\emptyset.$ Thus, the sub-additivity of $\phi$ entails that 
\[\left|W_\phi\left(uz^1\right)\right| \le \left|W_\phi\left(u_1u_2...u_{2r_1-2}\right)\right|+\left|W_\phi\left(u_{2r_1-1}u_{2r_1}z^1\right)\right| = \left|W_\phi\left(u\right)\right| - 2,\]
in other words the input of $z^1$ strictly decreases the size of the buffer content $u$, that is, if we let $u^{2}=W_\phi\left(uz^1\right)$, then $u^2$ is of even length $2r_2$, where $r_2 <r_1$. 
We then apply the same argument as above for $u^{2}$ instead of $u$: there exists an erasing word $z^2$ for the two-letter word $u^{2}_{2r_2-1}u^{2}_{2r_2}$ gathering the last two letters of $u^{2}$, so 
as above, 
\[\left|W_\phi\left(uz^1z^2\right)\right| = \left|W_\phi\left(u^{2}z^2\right)\right| \le \left|W_\phi\left(u^{2}\right)\right| - 2 .\] 
We can continue this construction by induction, until we reach an index $\ell$ such that 
\begin{equation}
\label{eq:OK}
W_\phi\left(uz^1z^2...z^\ell\right) = \emptyset.
\end{equation}
Observe that, as $z^1,...z^\ell$ are all erasing words, we have that $W_\phi(z^1)=W_\phi(z^2)=...=W_\phi(z^\ell)=\emptyset.$ 
Thus $W_\phi(z^1z^2...z^\ell)=\emptyset$, which shows, together with (\ref{eq:OK}), that $z=z^1z^2...z^\ell$ is an erasing word for $u$.

\section{Proof of assertion (ii) of Proposition \ref{pro:erasing2}}
\label{sec:appendixB}

We first need the following lemma, 
\begin{lemma}
\label{lemma:spanning}
Any connected non-bipartite graph $G=(\maV,\maE)$ can be spanned by an odd cycle, i.e. there exists 
an odd cycle $\maC=c_1\v c_2 \v \cdots \v c_{2q} \v c_{2q+1}\v c_1$, in which all the nodes of $\maV$ appear at least once.  
\end{lemma}

\begin{proof}
As $G$ is non-bipartite, $G$ contains an elementary odd cycle $\breve{\maC}:=k_1 \v k_2 \v \, \cdots \, \v k_{2p} \v k_{2p+1}\v k_1.$ 
Let $s \in \N$ be the number of nodes of $\maV$ which do not appear in the latter cycle, and denote by 
$i_1,...,i_s$, these nodes. By connectedness, there exists for any $j \in \llbracket 1,s \rrbracket$, a minimal path $\maP_j$ of length, say, $\ell_j$, 
from $k_1$ to $i_j$. Then, we can connect $k_1$ to itself by following, first, the cycle $\maC$, and then all the paths $\maP_j$ from $k_1$ to $i_j$ and then the reversed path of $\maP_j$ from $i_j$ to $k_1$, 
successively for all 
$j \in \llbracket 1,s \rrbracket$. The resulting path is a cycle connecting to $k_1$ to itself and spanning the whole set $\maV$, and its length is 
$$2p+1 + \sum_{j=1}^s 2\ell_j=2\left(p+\sum_{j=1}^s \ell_j\right)+1=:2q+1,$$ 
an odd number. 
\end{proof}

We can now prove assertion (ii) of Proposition \ref{pro:erasing2} for $G$ any non-bipartite graph, 
and $\phi=\textsc{lcfm}.$ First, from Lemma \ref{lemma:spanning}, there exists a cycle $\maC=c_1\v c_2 \v ... \v c_{2q+1}$ that spans 
the graph $G$. 
We then let $z^1$ be the word consisting of all the nodes of $\maC$ visited 4 times in that order, i.e. 
\[z^1=c_1c_2...c_{2q+1}c_1...c_{2q+1}c_1...c_{2q+1}c_1...c_{2q+1}.\]
We drop again the lists of permutations from all notation. 
Clearly, we have that 
\begin{equation}
\label{eq:voisC}
\maE\left(\left\{c_1,c_2,...,c_{2q+1}\right\}\right)=\maV,
\end{equation} 
and, as $\maC$ is a cycle, $W_{\textsc{lcfm}}(z)=\emptyset$, as for any admissible policy. Second, as $\maC$ is a cycle it is also clear that any prefix of $z^1$ of even size is completely matchable by any admissible policy. Thus, for any $w$
Now fix $i$ and $j$ in $\maV$ such that $i \pv j$. To show that $z^1$ is a $2$-strong erasing word, we need to check that $W_{\textsc{lcfm}}(ijz^1)=\emptyset$. 
For this let us define the following sets for $k \in \{i,j\}$, 
\begin{align*}
\mathcal H(k) &= \left\{\mbox{even indexes $2\ell$ in }\llbracket 1, 2q+1 \rrbracket\,:\,c_{2\ell} \v k \right\};\\
\mathcal O(k) &= \left\{\mbox{odd indexes $2\ell+1$ in }\llbracket 1, 2q+1 \rrbracket\,:\,c_{2\ell+1} \v k \right\}.
\end{align*}
We are in the following alternative: 
\begin{enumerate}
\item[Case 1:] $\mathcal O(i) \cup \mathcal O(j) \ne\emptyset$, i.e. $i$ or $j$ (or both) neighbor a node of odd index in $\maC$. Let $2p+1=\min \mathcal O(i) \cup \mathcal O(j)$. 
First observe that, by the definition of {\sc lcfm} all items of even indexes in $\llbracket 1,2p \rrbracket$ are matched with the immediate preceding item of odd index, so the entering 
$c_{2p+1}$ item finds only $i$ and $j$ in the system, and is matched with $j$ if $c_{2p+1} \v j$, or with $i$ if $j \pv c_{2p+1}$ and $i \v c_{2p+1}$. Let us assume that we are in the first case, the second one can be treated analogously. So we have $W_{\textsc{lcfm}}\left(ijc_1...c_{2p+1}\right)=i$. Let us now define  
\begin{equation*}
\tilde{\mathcal H}(i) = \left\{\mbox{even indexes $2\ell$ in }\llbracket 2p+2, 2q \rrbracket\,:\,c_{2\ell} \v i \right\}.
\end{equation*}
We have three sub-cases: 
     \begin{itemize}
     \item[Sub-case 1a:] $\tilde{\mathcal H}(i) \ne \emptyset$.\quad Set $2r=\min \tilde{\mathcal H}(i)$. Then the $i$ item is matched with $c_{2r}$. Indeed, in {\sc lcfm} all items of odd indexes in 
     $\llbracket 2p+2, 2r \rrbracket$ are matched with the immediate preceding item, even if they are compatible with $i$. 
     Then, after the $i$ item is matched with the $c_{2r}$ item, all items of odd indexes in $\llbracket 2r+1,2q-1 \rrbracket$ (if the latter is non-empty) in the first exploration of $\maC$ are matched 
     with the immediate following item, until the first $c_{2q+1}$ item is matched with the second $c_1$ item. After that, in the second exploration of $\maC$ all items of even nodes are matched with 
     the following item of odd index, until the second $c_{2q}$ item is matched with the second $2q+1$ item, so we get a perfect matching of $ij$ with the first two explorations of $\maC$. 
     Then the last two visits of $\maC$ are perfectly matched on the fly, since $\maC$ is a cycle. So $W_{\textsc{lcfm}}(ijz^1)=\emptyset.$ 
     \item[Sub-case 1b:] $\tilde{\mathcal H}(i)=\emptyset$ and $\mathcal O(i) \ne \emptyset$.\quad  
     Due to the {\sc lcfm} policy, in the first exploration of $\maC$ all odd items are matched with the immediate preceding 
     item of even index, until $c_{2q+1}$, in a way that $W_{\textsc{lcfm}}(ijc_1...c_{2q+1})=i$. Let $2s+1=\min \mathcal O(i)$. Then the remaining $i$ item is matched with the second $c_{2s+1}$, since  
     in {\sc lcfm}, all items of even indexes less than $2s+1$ that are compatible with $i$, are matched with the preceding item of odd index. After that, if $s<q$ then all remaining items of 
     even indexes in the second exploration of $\maC$ are matched with the immediate following item, until the second $c_{2q}$ item is matched with the second $c_{2q+1}.$ 
     Thus $W_\phi(ijz^1)=\emptyset$, and we conclude as in 1a. 
     \item[Sub-case 1c:] $\tilde{\mathcal H}(i)=\emptyset$ and $\mathcal O(i) = \emptyset$.\quad From (\ref{eq:voisC}) there necessarily exists an even index (take the smallest one) 
     $2u \in \llbracket 2,2p \rrbracket$ such that $i \v c_{2u}$. Then, as in 1b we have $W_{\textsc{lcfm}}\left(ijc_1...c_{2q+1}\right)=i$. Then, in the second exploration of $\maC$, in {\sc lcfm} all items of 
     even indexes are matched with the preceding item of odd index, until the second $c_{2q+1}$ remains unmatched, i.e. $W_{\textsc{lcfm}}\left(ijc_1...c_{2q+1}c_1...c_{2q+1}\right)=ic_{2q+1}$. 
     Then the remaining $c_{2q+1}$ item is matched with the third $c_1$, and in the third visit of $\maC$, all items of even indexes are matched with the following item of odd index, until 
     $c_{2u}$ is matched with $i$. To finish the third exploration, if $u<q$ then all items of odd index in $\llbracket 2u+1,2q-1 \rrbracket$ are matched with the following item of even index, until 
     the third $c_{2q+1}$ remains alone unmatched, i.e. $W_{\textsc{lcfm}}\left(ijc_1...c_{2q+1}c_1...c_{2q+1}c_1...c_{2q+1}\right)=c_{2q+1}$. At this point, the forth $c_1$ is matched with 
     the third $c_{2q+1}$, and then in the fourth exploration of $\maC$ all items of even index are matched with the following item of odd index, until the last $c_{2q}$ is matched with the last 
     $c_{2q+1}$ item. We end up again with $W_{\textsc{lcfm}}(ijz^1)=\emptyset.$ 
     \end{itemize}
\item[Case 2:] $\mathcal O(i) \cup \mathcal O(j) =\emptyset$.\quad In that case $i$ and $j$ both have only neighbors of even indexes in $\maC$, in particular from (\ref{eq:voisC}) $\maH(i)$ and $\maH(j)$ are both non-empty. Let $2p=\min \mathcal H(i)$ and $2p'=\min \mathcal H(j)$. Again from the definition of {\sc lcfm}, in the first exploration of $\maC$, all items of even indexes are matched with the preceding item of odd index, until $c_{2q+1}$ remains unmatched, so $W_{\textsc{lcfm}}\left(ijc_1...c_{2q+1}\right)=ijc_{2q+1}$. Then the first $c_{2q+1}$ item is matched with the second $c_1$, and if $2< \min(2p,2p')$, in the second exploration of $\maC$ all items of even index in $\llbracket 2,\min(2p,2p')-2  \rrbracket$ are matched with the following item of odd index. We have again, two sub-cases: 
    \begin{itemize}
    \item[Sub-case 2a:] $p'\le p$, so in {\sc lcfm} the $j$-item is matched with the second $c_{2p}$. 
    In the second exploration of $\maC$, after the $c_{2p}$ item has been matched with the $j$ item, if $p<q$ all items of odd indexes in 
    $\llbracket 2p+1,2q-1 \rrbracket$ are matched on the fly with the immediate following item of even index, until only the second $c_{2q+1}$ item remains unmatched, 
    so $W_{\textsc{lcfm}}\left(ijc_1...c_{2q+1}c_1...c_{2q+1}\right)=ic_{2q+1}$. Then the second $c_{2q+1}$ is matched with the third $c_1$. In the third exploration, if $p>2$, all items of even indexes 
    in $\llbracket 2,2p-2 \rrbracket$ are matched with the following item, until the $c_{2p}$ item is matched with $i$. 
    We then conclude as in 1c, and end up again with $W_{\textsc{lcfm}}(ijz^1)=\emptyset.$ 
    \item[Sub-case 2b:] $p < p'$, so the $i$-item is matched with the second $c_{2p}$. Then the $j$ item remains to be matched, and we conclude exactly as in 2a, by matching the $j$ item with the 
    third $c_{2p'}$ (instead of $i$ with the third $c_{2p}$). This concludes the proof. 
    \end{itemize}   
 \end{enumerate}

\section{Proof of (iii) of Proposition \ref{pro:erasing2}}
\label{sec:appendixC}
We now turn to the proof of (iii) of Proposition \ref{pro:erasing2}. Take $G$, the odd cycle $c_1\v c_2 \v ... \v c_{2q+1}$ and $\phi=\textsc{fcfm}$. We set 
$$z=\vec z \,\vec z \,\cv z \,\cv z:=c_1c_2\cdots c_{2q}c_{2q+1}c_1c_2\cdots c_{2q}c_{2q+1}c_1c_{2q+1}c_{2q}\cdots c_2 c_1 c_{2q+1} c_{2q} \cdots c_2,$$ namely, twice the word $\vec z$ exploring the cycle in increasing order, and then twice the word $\cv z$ exploring the cycle in decreasing order. 

Let us first take the case of a two-letter word $c_ic_j$, where $i,j$ are two nodes of $G$ such that $i<j$ and $i\pv j$. Then we have the following sub-cases:
\begin{itemize}
\item[Case 1]: $i$ is even, $j$ is odd. In that case, in $\Mfcfs(ijz)$, the first incoming $c_1$-item is matched with the first incoming $c_2$-item, the first $c_3$ with the first $c_4$, and so on, until the first $c_{i-1}$-item is matched with the stored $c_i$-item. Then, the first incoming $c_i$-item is matched with the first incoming $c_{i+1}$-item, the first $c_{i+2}$ with the first $c_{i+3}$, and so on, until the first incoming $c_{j-1}$-item (remember that $j-1$ is even) is matched with the stored $c_j$-item. Then all subsequent letters (from the first incoming $c_j$ item to the end of the second word $\cv z$), are immediately matched 2 by 2 in order of arrivals. As they is an even number of such letters, all letters have been matched and we get 
$\Mfcfs(c_ic_jz)=\emptyset$.
\item[Case 2]: $i$ and $j$ are both even. Then, as in Case 1, in $\Mfcfs(c_ic_jz)$ the first incoming $c_1$ is matched with the first incoming $c_2$, the first $c_3$ with the first $c_4$, and so on, until the first $c_{i-1}$-item is matched with the stored $c_i$-item. Then (if $j\ge i+4$), the first incoming $c_i$-item is matched with the first incoming $c_{i+1}$-item, the first $c_{i+2}$ with the first $c_{i+3}$, and so on, until the first incoming $c_{j-4}$ is matched with the first incoming $c_{j-3}$. Then, in {\sc fcfm}, the first incoming $c_{j-1}$ is matched with the stored $c_j$-item, so the first incoming $c_{j-2}$ (which is precisely the first incoming  $c_i$-item if $j=i+2$) remains unmatched. Then, the first incoming $c_j$ gets matched with the first incoming $c_{j+1}$, and so on, until the end of the first $\vec z$, where $c_{2q}$ is matched with $c_{2q+1}$. We get that $\Wfcfs(c_ic_j\vec z)=c_{j-2}$. 
Then, in the second $\vec z$, all letters of even indexes of the second $\vec z$ are matched with the immediately preceding item of odd index, until the second $c_{j-4}$ is matched with the second $c_{j-5}$. 
Then, the second $c_{j-3}$ is matched with the stored $c_{j-2}$ of the first $\vec z$. 
Then, as $j-2$ is even, we conclude the second $\vec z$ by matching items of  odd indexes with the immediately preceding item of even index, until the second $c_{2q+1}$ is matched with the second $c_{2q}$. We thus get $\Wfcfs(c_ic_j\vec z\,\vec z)=\emptyset$ and thus, clearly, $$\Wfcfs(c_ic_jz)=\Wfcfs(\cv z\,\cv z)=\emptyset.$$
\item[Case 3]: $i$ is odd and $j$ is even. Then, 
\begin{itemize}
\item[Sub-case 3a]: $i\ge 3$. In that case, in the first $\vec z$ all letters of even index are matched with the preceding item of odd index, until $c_{i-3}$ is matched with $c_{i-4}$ (if $i\ge 5$). Then, $c_{i-1}$ is matched with the stored $c_{i}$-item, and then the first incoming $c_i$ is stored with the successive $c_{i+1}$, and then (if $j\ge i+4$), all elements of $\vec z$ of even indexes are matched with the immediately preceding element of odd index, until $c_{j-2}$ is matched with $c_{j-3}$. 
Then, $c_{j-1}$ is matched with the stored $c_j$-item, and then, until the end of the first $\vec z$, items of odd indexes are matched with the immediately preceding element of even index. So $$\Wfcfs(c_ic_j\vec z)=c_{i-2}.$$ 
Then, 
	\begin{itemize}
	\item If $i\ge 7$, in the second $\vec z$, elements of even indexes are matched with the immediately preceding 	
	element of odd index, until $c_{i-5}$ is matched with $c_{i-6}$. Then the second $c_{i-3}$ is matched with 
	the stored $c_{i-2}$, whereas $c_{i-4}$ remains at first unmatched. Then, all elements 
	of even index are matched with the immediately preceding element of odd index, until $c_{2q}$ is matched 
	with $c_{2q-1}$. So we get 
		\begin{equation}
		\label{eq:CIP0}
		\Wfcfs(c_ic_j\vec z\,\vec z)=c_{i-2}c_{2q+1}.
		\end{equation}
	We then show the following result: 
	\begin{equation}
	\label{eq:CIP1}
	\mbox{For any odd }k\ge 3,\quad  \Wfcfs(c_kc_{2q+1}\cv z\,\cv z)=\emptyset.
	\end{equation}
	Indeed, in the first $\cv z$, $c_1$ is matched with the stored $c_{2q+1}$, and then any element of even index $\ell$ 
	is matched with the immediately preceding element of odd index $\ell+1$, until $c_{k+3}$ is matched with 
	$c_{k+4}$ (if $k\le 2q-3$). Then, the first $c_{k+1}$ is matched with the stored $c_k$, and $c_{k+2}$ remains at 
	first unmatched. Then, in the first $\cv z$ all elements of even index $\ell$ are matched with the element of odd index 
	$\ell+1$, until $c_1$ remains unmatched. So we get  $\Wfcfs(c_kc_{2q+1}\cv z)=c_{k+2}c_1$. Last, in the second 
	$\cv z$, $c_{2q+1}$ is matched with the stored $c_1$, and then each element of odd index $\ell$ is matched with the 
	immediately preceding element of even index $\ell+1$, until (if $k\le 2q-5$) 
	the second $c_{k+4}$ gets matched with the second $c_{k+5}$. Then, the second $c_{k+3}$ gets matched with the 
	stored $c_{k+2}$, and then, in the remainder of the second $\cv z$ all elements of even index $\ell$ 
	are matched with the immediately preceding element of odd index $\ell+1$, until $c_2$ is matched with $c_3$. 
	We finally obtain that $\Wfcfs(c_kc_{2q+1}\cv z\,\cv z)=\emptyset$, which completes the proof of \eqref{eq:CIP1}. 
	
	Last, from \eqref{eq:CIP0} and \eqref{eq:CIP1}, we obtain that 
	\[\Wfcfs(c_ic_j\vec z\,\vec z\,\cv z\,\cv z)=\Wfcfs(c_{i-2}c_{2q+1}\,\cv z\,\cv z)=\emptyset.\]
	\item If $i=5$, the second $c_{2}$ is matched with the stored $c_{3}$. Then, in the second $\vec z$, $c_{1}$ 
	remains at first unmatched, and then all elements of even index are matched with the immediately preceding 
	element of odd index, until $c_{2q}$ is matched with $c_{2q-1}$. 
	Then the second $c_{2q+1}$ gets matched with the stored $c_1$, and so $\Wfcfs(c_ic_j\vec z\,\vec z)=\emptyset$, 
	implying in turn that 
	\[\Wfcfs(c_ic_j\vec z\,\vec z\,\cv z\,\cv z)=\Wfcfs(\cv z\,\cv z)=\emptyset.\]
	\item If $i=3$, then, for the first $c_2$ is matched with the stored $c_2$, and then in $\vec z$, any element of even 
	index until $c_{j-2}$ is matched with the immediate preceding item, until $c_{j-1}$ is matched with the stored 
	$c_j$. 
	Then, in the first $\vec z$, any item of odd index is matched with the immediate preceding item of even index, until 
	$c_{2q+1}$ is matched with $c_{2q}$. So $\Wfcfs(c_3c_j\,\vec z)=c_1$. Then, in the second $\vec z$, $c_2$ is 
	matched with the stored $c_1$, the second incoming $c_1$ remains at first unmatched, and then in the second 
	$\vec z$, any item of event index is matched with the immediate preceding item, until $c_{2q}$ is matched with 
	$c_{2q-1}$. Finally, the last $c_{2q+1}$ is matched with the stored $c_1$, so 
	$\Wfcfs(c_3c_j\,\vec z\,\vec z)=\emptyset$, and thus 
	\[\Wfcfs(c_3c_j\,z)=\Wfcfs(\cv z\,\cv z)=\emptyset.\]
	\end{itemize}
\item[Subcase 3b]: $i=1$. Then in $\vec z$, $c_2$ is matched with the stored $c_1$, and then any item of even index 
until $c_{j_2}$ is matched with the immediate preceding item of odd index, until $c_{j-1}$ is matched with the stored $c_j$-item. Then, exactly as is in the sub-sub-case above for $i=3$, we get that 
$\Wfcfs(c_1c_j\,\vec z\,\vec z)=\emptyset$ and thus $\Wfcfs(c_1c_j\,z)=\emptyset.$ 
\end{itemize}
\item[Case 4]: $i$ and $j$ are both odd. Then (if $i\ge 3$), in the first $\vec z$ all elements of even index until $c_{i-3}$ are matched with the immediate preceding item. Then, $c_{i-1}$ is matched with the stored $c_i$, and so $c_{i-2}$ is at first unmatched. Then (if $j\ge i+4$), all items of even index until $c_{j-3}$ are matched with the immediate preceding item. Then, $c_{j-2}$ is at first unmatched, whereas $c_{j-1}$ with the stored $c_j$. Then, all remaining elements of $\vec z$ are matched right away two by two until $c_{2q-1}$ is matched with $c_{2q}$. So $\Wfcfs(c_ic_j\,\vec z)=c_{i-2}c_{j-2}c_{2q+1}$ Then, in the second $\vec z$, $c_1$ is matched with the stored $c_{2q+1}$, and then all pairs of elements (if any) between $c_2$ and $c_{i-4}$ are matched two-by-two. Then, $c_{i-1}$ is matched with the stored $c_{i-2}$, and then all pairs of elements (if any) between $c_i$ and $c_{j-5}$ are matched two-by-two. 
Then $c_{j-4}$ remains unmatched, and $c_{j-3}$ is matched with the stored $c_{j-2}$. To finish the second $\vec z$, all elements between $c_{j-2}$ and $c_{2q}$ are stored two-by-two, and so we obtain that 
$\Wfcfs(c_ic_j\,\vec z\,\vec z)=c_{j-4}c_{2q+1}$. But as $j-4$ is an odd number, we can apply \eqref{eq:CIP1}, to conclude that \[\Wfcfs(c_ic_jz)=\Wfcfs(c_{j-4}c_{2q+1}\,\cv z\,\cv z)=\emptyset.\]
\end{itemize}
As a first conclusion, we obtain that for all $i<$ such that $c_i\pv c_j$, 
$\Wfcfs(c_jc_i\,z)=\emptyset.$ 
But it is immediate that whenever $i$ and $j$ are such that $j<i$ and $c_i\pv c_j$, 
we also have that 
\[\Wfcfs(c_ic_j\,z)=\Wfcfs(c_jc_i\,z)=\emptyset.\]
Last, it can be proved exactly as in Case 2 that if $i$ is even $\Wfcfs(c_ic_i\,\vec z\vec z)=\emptyset$, and 
so 
$\Wfcfs(c_ic_i\,z)=\emptyset$. All the same, exactly as in case 4, we get that if $i$ is odd, 
$\Wfcfs(c_ic_i\,\vec z\vec z)=c_kc_{2q+1}$ for some off index $k$. Using \eqref{eq:CIP1}, we conclude again that $\Wfcfs(c_ic_i\,z)=\emptyset$, which concludes the proof.

\end{document}